\documentclass{amsart}
\usepackage{amsmath,amssymb, amsthm}
\usepackage{bbold}
\usepackage{cite}
\usepackage{physics}
\usepackage{tikz-cd}
\usepackage{tikz}
\usepackage{pst-solides3d}
\usepackage{float}
\usetikzlibrary{positioning}
\usetikzlibrary{arrows}
\usetikzlibrary{calc}
\usetikzlibrary{shapes}
\usetikzlibrary{decorations.pathreplacing}
\tikzstyle{arrow} = [thick,->,>=stealth]
\usepackage[margin=1.32in]{geometry}

\newtheorem{theorem}{Theorem}[section]
\newtheorem{proposition}[theorem]{Proposition}
\newtheorem{lemma}[theorem]{Lemma}
\newtheorem{corollary}[theorem]{Corollary}

\numberwithin{equation}{section}
\title{Overlapping Discs Filtration in the Commutative Operad}
\author{Keely Grossnickle}
\address{Department of Mathematics\\
	Kansas State University\\
	Manhatan, KS 66506, USA}
\email{kgrossni@ksu.edu}
\date{}

\begin{document}

	\begin{abstract}
		The spaces of configurations of non-$k$-overlapping discs have been studied as a bimodule over the little discs operad. In fact, the spaces form a filtered operad. We define and study the induced structure on the homology.
	\end{abstract}
	\maketitle
	\sloppy
		\section{Introduction}
	Let $\mathcal{M}_{d}(n)$ be the configuration space of $n$ distinct, labeled points in $\mathbb{R}^{d}$. We can impose a $non-k-equal$ condition such that no $k$ coincide, denote this space by $\mathcal{M}_{d}^{(k)}(n)$. Let $\mathcal{B}_{d}$ denote the little $d$-discs operad, where $\mathcal{B}_{d}(n)$ is the configuration space of $n$ open, labeled, disjoint discs inside the unit disc. Equivalently, $\mathcal{B}_{d}(n)=sEmb\left(\bigsqcup\limits_{n} D^{d}, D^{d} \right)$, the space of special embeddings, $\bigsqcup\limits_{n}D^{d}\hookrightarrow D^{d}$ given on each component by translation and rescaling only. We can impose a $non-k-overlapping$ condition on a configuration of discs where no $k$ discs share a common point. Denote this new space by $\mathcal{B}_{d}^{(k)}(n)$. One can view $\mathcal{B}_{d}^{(k)}(n)$ as a space $sImm^{(k)}\left(\bigsqcup\limits_{n} D^{d},D^{d}\right)$, the space of non-$k$-overlapping special immersions, $\bigsqcup\limits_{n}D^{d} \looparrowright D^{d}$ given on each component by translation and rescaling only. There is a homotopy equivalence $\mathcal{B}_{d}^{(k)}(n)\rightarrow \mathcal{M}_{d}^{(k)}(n)$ by taking the center of each disc. In this paper we will mainly use $\mathcal{B}_{d}^{(k)}$ since it gives us more structure, however $\mathcal{M}_{d}^{(k)}$ is more useful when proving Theorem \ref{maintheorem} of the paper. 
	
	More general configuration spaces have been studied in \cite{kosar2016cohomology}, which produced similar relations to the ones found in this paper.
	
	The sequence $\mathcal{B}_{d}^{(k)}$ is a bimodule over $\mathcal{B}_{d}$. The right action 
	\begin{equation} \label{bimodoverlittlediscrightaction}
	\mathcal{B}_{d}^{(k)}(n)\times \mathcal{B}_{d}(m_{1})\times \mathcal{B}_{d}(m_{2})\times...\times\mathcal{B}_{d}(m_{n})\rightarrow \mathcal{B}_{d}^{(k)}(m_{1}+m_{2}+...+m_{n})
	\end{equation}
	and the left action
	\begin{equation}\label{bimodoverlittlediscleftaction}
	\mathcal{B}_{d}(n)\times \mathcal{B}_{d}^{(k)}(m_{1})\times \mathcal{B}_{d}^{(k)}(m_{2})\times...\times\mathcal{B}_{d}^{(k)}(m_{n})\rightarrow \mathcal{B}_{d}^{(k)}(m_{1}+m_{2}+...+m_{n})
	\end{equation}
	are given by composition of disc maps. The composition maps are well-defined as the resulting configurations satisfy the non-$k$-overlapping condition.
	
	Recall that a pointed bimodule under $H_{\ast} \mathcal{B}_{d}$ is a bimodule over $H_{\ast} \mathcal{B}_{d}$ containing $\mathcal{C}om$, the commutative operad. That is, one has a commutative diagram of $H_{\ast}\mathcal{B}_{d}$-bimodules.
	\begin{center}
		\begin{tikzpicture}
		[nodes={fill=black!0}]
		\node (topleft) [ ] {$H_{\ast} \mathcal{B}_{d}$};
		\node (topright) [right of=topleft, xshift=1cm] {$H_{\ast} \mathcal{B}_{d}^{(k)}$};
		\node (bottom) [below of=topleft, xshift=1cm] {$\mathcal{C}om$};
		\draw [->] (topleft) -- (topright);
		\draw [->] (topleft) -- (bottom);
		\draw [right hook->] (bottom) -- (topright);
		\end{tikzpicture}
	\end{center}
	
	Also recall that $H_{\ast}\mathcal{B}_{d}$ is the associative unital operad for $d=1$ and the graded Poisson unital operad, with bracket of degree $d-1$ for $d\geq 2$, see \cite{cohen}.
		
	In \cite{D_T}, the authors describe $H_{\ast}\mathcal{B}_{d}^{(k)}$ as a pointed bimodule under $H_{\ast}\mathcal{B}_{d}$ in Theorem 3.6 of their paper.
	
	\begin{theorem} [N. Dobrinskaya and V. Turchin, \cite{D_T}] \label{dttheorem}
		For $k \geq 3$, the pointed bimodule  $H_{\ast}\mathcal{B}_{d}^{(k)}$ under $H_{\ast}\mathcal{B}_{d}$ is generated by one element: $\{x_{1},...,x_{k}\}\in H_{(k-1)d-1}\mathcal{B}_{d}^{(k)}(k)$ satisfying the following properties:
		symmetric or skew symmetric depending on the parity of $d$:
		\begin{equation} \label{thmsymmrelation}
		\{x_{\sigma_{1}}...x_{\sigma_{k}}\}=(-1)^{|\sigma|d}\{x_{1},...,x_{k}\},\ \sigma \in \Sigma_{k}.
		\end{equation}
		The only relation on the left action is the generalized Jacobi: 
		\begin{equation} \label{thmjacobirelation}
		\sum_{i=1}^{k+1}(-1)^{(i-1)d}[x_{i},\{x_{1},...,\hat{x}_{i},...,x_{k+1}\}]=0,
		\end{equation}	
		 and lastly two Leibniz relations with respect to the right action:
		\begin{equation} \label{thmleibniz1}
		\{x_{1},...,x_{k-1}, x_{k}\cdot x_{k+1}\} = x_{k}\cdot \{x_{1},...,x_{k-1},x_{k+1}\} + \{x_{1},...,x_{k}\}\cdot x_{k+1};
		\end{equation}
		\begin{equation} \label{thmleibniz2}
		\{x_{1},...,x_{k-1}, [x_{k},x_{k+1}]\} = (-1)^{d}[\{x_{1},...,x_{k-1},x_{k+1}\},x_{k}]+[\{x_{1},...,x_{k}\}, x_{k+1}].
		\end{equation}
	\end{theorem}
		
	As a bimodule $H_{\ast}\mathcal{B}_{d}^{(k)}(k)$ has two generators, $\{x_{1},...,x_{k}\}\in H_{(k-1)d-1}$ and the generator $x_{1}\in H_{0}\mathcal{B}_{d}^{(k)}(1)$ of $\mathcal{C}$om $=H_{0}\mathcal{B}_{d}^{(k)}$. Below we give a geometrical description of $\{x_{1},...,x_{k}\}$.  We know $$\mathcal{B}_{d}^{(k)}(k)\simeq \{(x_{1},...,x_{k}) \in (\mathbb{R}^{d})^{k}|\ \mathrm{exclude} \  x_{1}=x_{2}=...=x_{k}\} = \mathbb{R}^{dk} - \mathbb{R}^{d} \simeq S^{(k-1)d-1}.$$
	This implies that the element $\{x_{1},...,x_{k}\} \in H_{(k-1)d-1}\mathcal{B}_{d}^{(k)}(k)$ can geometrically be realized as a $[(k-1)d-1]$-sphere:
	\begin{equation}\label{sphererepresentation}
	|x_{1}|^{2}+|x_{2}|^{2}+...+|x_{k}|^{2} = \varepsilon^{2} \qquad \qquad \sum_{i=1}^{k} x_{i} = 0
	\end{equation}
	\noindent Here, $x_{i}$ stands for the center of the $i$-th ball, each of radius $<<\varepsilon$.
	
	Elements in $H_{\ast}\mathcal{B}_{d}^{(k)}(n), d\geq 2$ can be thought as products of iterated brackets. Formally, one can have a bracket, $[\cdot,\cdot]$, or a multiplication inside of a brace, however by the two Leibniz relations \eqref{thmleibniz1} and \eqref{thmleibniz2}, they can be pulled outside of the brace. This means the left action on the two generators span all of the homology and nothing new comes from the right action. All braces must be of the same length and braces cannot be inside of other braces. Singletons are allowed in the products of brackets. One must have at least one brace inside a bracket, that is, $[[x_{1},x_{3}],x_{2}]$ is zero as an element. It is possible to have two or more braces inside a bracket, for example $[\{x_{1},x_{3},x_{5}\},\{x_{2},x_{4},x_{6}\}]$.
		
	\textbf{Example:} Let us look at an example of an element in $H_{\ast}\mathcal{B}_{d}^{(k)}(n)$ for $k=3$, $n=5$. Let us examine the element $x_{2}\cdot [\{x_{1}, x_{3}, x_{4}\}, x_{5}] \in H_{3d-2}\mathcal{B}_{d}^{(3)}(5)$. This element can be geometrically represented as the following product of spheres: $S^{(k-1)d-1}\times S^{d-1} = S^{2d-1}\times S^{d-1}$, since $k=3$. Here $x_{1}, x_{3}, x_{4}$ orbit closely around each other making a spherical class as seen above in \eqref{sphererepresentation}. The disc $x_{5}$ orbits around $x_{1},x_{3},x_{4}$. Lastly, $x_{2}$ does not interact with the other discs and stays still far away from the other points. 
	
	
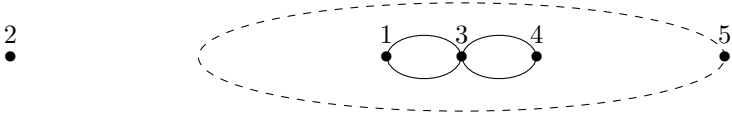
\begin{tikzpicture}
	\node (X2) [draw=none, fill=none] { };
	\node (X1) [draw=none, fill=none, right of=X2, xshift=4cm] { };
	\node (X3) [draw=none, fill=none, right of=X1] { };
	\node (X4) [draw=none, fill=none, right of=X3] { };
	\node (X5) [draw=none, fill=none, right of=X4, xshift=1.5cm] { };
	
	\node (number2) [draw=none, fill=none, below of=X2, yshift=0.3cm] {2};
	\node (number1) [draw=none, fill=none, below of=X1, yshift=0.3cm] {1};
	\node (number3) [draw=none, fill=none, below of=X3, yshift=0.3cm] {3};
	\node (number4) [draw=none, fill=none, below of=X4, yshift=0.3cm] {4};
	\node (number5) [draw=none, fill=none, below of=X5, yshift=0.3cm] {5};
	
	\node (dot2) [draw=none, fill=none, below of=X2] {$\bullet$};
	\node (dot1) [draw=none, fill=none, below of=X1] {$\bullet$};
	\node (dot3) [draw=none, fill=none, below of=X3] {$\bullet$};
	\node (dot4) [draw=none, fill=none, below of=X4] {$\bullet$};
	\node (dot5) [draw=none, fill=none, below of=X5] {$\bullet$};
	
	\draw (5,-1) to [out=80, in=100] (6,-1);
	\draw (5,-1) to [out=-80, in=-100] (6,-1);
	\draw (6,-1) to [out=80, in=100] (7,-1);
	\draw (6,-1) to [out=-80, in=-100] (7,-1);
	\draw [dashed] (6,-1) ellipse (3.5cm and 0.72cm);
	\end{tikzpicture}
	
	\vspace{0.2cm}
	
	Theorem \ref{dttheorem} still holds when $d=1$. However the bracket $[x_{1},x_{2}]$ should be understood as $x_{1}x_{2}-x_{2}x_{1}$ so the only operation is multiplication, as the underlying operad $H_{\ast}\mathcal{B}_{1}$ is $\mathcal{A}ssoc$, the associative operad. In this case \eqref{thmleibniz1} implies \eqref{thmleibniz2}. The relation \eqref{thmjacobirelation} is instead equivalently written as $$\sum_{i=1}^{k+1}(-1)^{i-1} \left( x_{i}\cdot \{x_{1},...,\hat{x}_{i},...,x_{k+1}\}-\{x_{1},...,\hat{x}_{i},...,x_{k+1}\}\cdot x_{i}\right).$$ 
	
	In this paper we introduce and study the filtered operad of overlapping discs, denoted $\mathcal{B}_{d}^{(\infty)}$. This new operad will contain information on both the operad $\mathcal{B}_{d}$ and all the bimodules $\mathcal{B}_{d}^{(k)}$. Its filtration is natural by the degree of the overlap. The main results are Theorem \ref{maintheorem}, Lemma \ref{lemmaoverlappingdiscinclusionnullhom}, Corollary \ref{corollary}, Proposition \ref{lemmacompositionstrivial}, and Lemma \ref{newleibnizeqs} which describe the induced structure in the homology that comes from the structure of a filtered operad on $\mathcal{B}_{d}^{(\infty)}$
	
	\subsection{Motivation}
	
	One application of the computations that we have in mind is for the study of spaces of immersions. Let $Imm_{\partial}^{(k)}(\mathbb{D}^{m},\mathbb{D}^{n})$ be the space of immersions of $\mathbb{D}^{m}\hookrightarrow\mathbb{D}^{n}$ of discs that are the standard inclusion in a neighborhood of the boundary and satisfy the condition that the image of any $k$-element subset has more than one point. Such spaces are called non-$k$-equal immersions. The bimodules of non-$k$-overlapping discs naturally appear in the study of these spaces \cite[Section 11]{D_T}.
	
	One has a natural sequence of inclusions:
	\begin{equation} \label{filteredimmersions}
	\overline{Emb}_{\partial} (\mathbb{D}^{m},\mathbb{D}^{n}) = \overline{Imm}_{\partial} {}^{(2)}(\mathbb{D}^{m},\mathbb{D}^{n}) \subset \overline{Imm}_{\partial}{}^{(3)}(\mathbb{D}^{m},\mathbb{D}^{n}) \subset \cdots \subset\overline{Imm}_{\partial}{}^{(\infty)}(\mathbb{D}^{m},\mathbb{D}^{n}), 
	\end{equation}
	where $$\overline{Imm}_{\partial}{}^{(k)}(\mathbb{D}^{m},\mathbb{D}^{n})= hofib\left(Imm_{\partial}^{(k)}(\mathbb{D}^{m},\mathbb{D}^{n})\rightarrow Imm_{\partial}(\mathbb{D}^{m}, \mathbb{D}^{n})\right).$$
	Note that $\overline{Imm}_{\partial}{}^{(\infty)}(\mathbb{D}^{m},\mathbb{D}^{n})= \overline{Imm}_{\partial}(\mathbb{D}^{m},\mathbb{D}^{n}) \simeq \ast.$
	Each space $\overline{Imm}_{\partial}{}^{(k)}(\mathbb{D}^{m},\mathbb{D}^{n})$ is naturally a $\mathcal{B}_{m}$-algebra. Moreover $\overline{Imm}_{\partial}{}^{(2)}(\mathbb{D}^{m},\mathbb{D}^{n})=\overline{Emb}_{\partial} (\mathbb{D}^{m},\mathbb{D}^{n})$ is a $\mathcal{B}_{m+1}$-algebra \cite{budneycubes,budneysplice, sakai}. (One dimension higher little discs action comes from the fact that embeddings can be pulled through each other.) Exactly the same construction endows \eqref{filteredimmersions} with a structure of a $filtered$ $\mathcal{B}_{m+1}-algebra$. This in particular means that one has a Browder bracket:
	\begin{equation}\label{browderbracket}
	[\cdot,\cdot]: H_{i}\overline{Imm}_{\partial}{}^{(k_{1})}(\mathbb{D}^{m},\mathbb{D}^{n})\times H_{j}\overline{Imm}_{\partial}{}^{(k_{2})}(\mathbb{D}^{m},\mathbb{D}^{n}) \rightarrow H_{i+j+m}\overline{Imm}_{\partial}{}^{(k_{1}+k_{2}-2)}(\mathbb{D}^{m},\mathbb{D}^{n}).
	\end{equation}
	
	The Goodwillie-Weiss tower associated to the space $\overline{Imm}_{\partial}{}^{(k)}(\mathbb{D}^{m},\mathbb{D}^{n})$ is expressed in terms of an (infinitesimal) $\mathcal{B}_{m}$-bimodule $\mathcal{B}_{n}^{(k)}$, \cite[Section 11]{D_T}. This tower is conjectured to converge to $\overline{Imm}_{\partial}{}^{(k)}(\mathbb{D}^{m},\mathbb{D}^{n})$, provided $n-m\geq 2$. We believe that the structure of a filtered operad on $\mathcal{B}_{n}^{(\bullet)}$ and the computations made in this paper can be used to understand the structure of a filtered $\mathcal{B}_{m+1}$-algebra \eqref{filteredimmersions} and in particular the Browder bracket operator \eqref{browderbracket}.
	
	\subsection{Acknowledgments}
	
	The author is thankful to her advisor, V. Turchin, for his time and guidance, particularly in determining the signs for Theorem \ref{maintheorem}.
	
	\section{Filtered Operads of Overlapping Discs} \label{sectionfilteredoperad}
	
	An operad $\mathcal{O}$ is filtered if there is a filtration in each component of $F_{0}\mathcal{O}(n)\subset F_{1}\mathcal{O}(n)\subset F_{2}\mathcal{O}(n)\subset...$ compatible with the composition maps:
	\begin{equation} \label{filteredoperadcomp}
	\circ_{i}: F_{k_{1}}\mathcal{O}(n_{1}) \times F_{k_{2}}\mathcal{O}(n_{2}) \rightarrow F_{k_{1}+k_{2}}\mathcal{O}(n_{1}+n_{2}-1).
	\end{equation}
	
	We assume that $id\in F_{0}\mathcal{O}(1)$. Note that $F_{0}\mathcal{O}$ is a suboperad of $\mathcal{O}$.
	
	The filtration of the operad induces a sequence of maps in homology:
	\begin{equation} \label{inducedsequencehom}
	H_{\ast}F_{0}\mathcal{O}(n)\rightarrow H_{\ast}F_{1}\mathcal{O}(n)\rightarrow H_{\ast}F_{2}\mathcal{O}(n)\rightarrow...
	\end{equation}
	and the composition maps:
	\begin{equation} \label{filteredoperadhomcomp}
	\circ_{i}: H_{\ast}F_{k_{1}}\mathcal{O}(n_{1}) \otimes H_{\ast}F_{k_{2}}\mathcal{O}(n_{2}) \rightarrow H_{\ast}F_{k_{1}+k_{2}}\mathcal{O}(n_{1}+n_{2}-1).
	\end{equation}
	
	One has the inclusion $\mathcal{B}_{d}^{(k)}(n) \subset \mathcal{B}_{d}^{(k+1)}(n)$ since the non-$k$-overlapping condition is stricter than the non-$(k+1)$-overlapping condition.
	Now define $F_{i}\mathcal{O}(n) := \mathcal{B}_{d}^{(i+2)}(n)$, where $\mathcal{O}(n) := 
	\mathcal{B}_{d}^{(\infty)}(n):=\bigcup\limits_{i=2}\limits^{\infty}\mathcal{B}_{d}^{(i)}(n)$. Just as with $\mathcal{B}_{d}$, here composition \eqref{filteredoperadcomp} is inserting a configuration of $n_{2}$ discs from $\mathcal{B}_{d}^{(k_{2})}(n_{2})$ into the $i$-th disc of $\mathcal{B}_{d}^{(k_{1})}(n_{1})$. When $k_{1}=k_{2}=2$ we get the usual operadic composition in $\mathcal{B}_{d}$. Note that $\mathcal{B}_{d}^{(2)}$ is the usual little discs operad $\mathcal{B}_{d}$. When $k_{1} = 2$ we get the infinitesimal version of the left action \eqref{bimodoverlittlediscleftaction} and when $k_{2} = 2$ we get the infinitesimal version of the right action \eqref{bimodoverlittlediscrightaction}. Note that $\mathcal{B}_{d}^{(\infty)}(n)$ = $\left( \mathcal{B}_{d}(1)\right)^{n}$ and therefore is contractible. Thus $\mathcal{B}_{d}^{(\infty)}$ is equivalent to the commutative operad.
	
	From \cite{D_T}, we already know the homology groups of $H_{\ast}\mathcal{B}_{d}^{(k)}(n)$, $k\geq 2$, and how the composition maps work when either $k_{1}$ or $ k_{2}$ is 2. Now we want to understand the maps in the sequence \eqref{inducedsequencehom} as well as the composition maps from \eqref{filteredoperadhomcomp}, for $k_{1},k_{2}>2$. To understand the sequence \eqref{inducedsequencehom} we will need the following lemma, which will also be useful when understanding the composition maps \eqref{filteredoperadhomcomp}.
\begin{lemma}\label{lemmaoverlappingdiscinclusionnullhom}
	For all $d\geq 1$, $k\geq 2$ and $n\geq 0$, the inclusion $\mathcal{B}_{d}^{(k)}(n) \subset \mathcal{B}_{d}^{(k+1)}(n)$ is null-homotopic.
\end{lemma}

\begin{proof}
	Define a homotopy $H: B_{d}^{(k)}(n)\times[0,1]\rightarrow B_{d}^{(k+1)}(n)$. Subdivide $[0,1]$ into $n+1$ subintervals. Recall that a point in $\mathcal{B}_{d}^{(k)}(n)$ is a configuration of $n$-discs in the unit disc with the condition that the intersection of any $k$ of them is empty. Fix a point in $\mathcal{B}_{d}^{(k+1)}(n)$ where all the discs, labeled $1',2',...,n'$, are disjoint. We will call this configuration the standard position for the discs. Now take any point $P \in \mathcal{B}_{d}^{(k)}(n)$, discs labeled $1,2,...,n$. Recall that $P$ lies inside of a unit disc. We can smoothly rescale and translate this unit disc so that it is disjoint from the $1',...,n'$ discs in the standard configuration. This homotopy is done on the first subinterval $\left[0,\frac{1}{n+1}\right]$.
	
	Next, we can smoothly rescale and translate the disc labeled 1 in $P$ to the disc labeled $1'$ in the standard position during the second interval of $H$. Then we can smoothly rescale and translate the disc labeled 2 in $P$ to the disc labeled $2'$ in the standard position during the third interval of $H$. We can iteratively do this for all $n$ discs in $P$ until each disc is in the standard position in $\mathcal{B}_{d}^{(k+1)}(n)$. In the $i$-th interval of $H$, rescale and translate the  disc labeled $(i-1)$ to the standard position, for $i\geq 2$. 
	
	Note when moving the discs, up to $k$ overlaps can occur. However since the $k$ overlaps are allowed in ${B}_{d}^{(k+1)}(n)$, the homotopy is well-defined.

We show the idea of the above proof below pictorially for $\mathcal{B}_{d}^{(2)}(3) \subset \mathcal{B}_{d}^{(3)}(3)$.
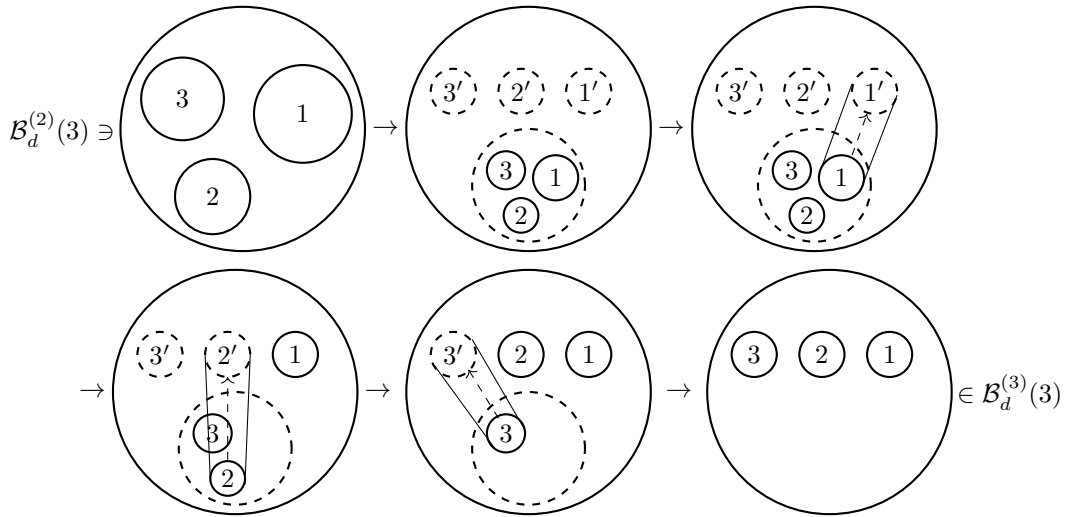
\begin{figure}[H]
\begin{tikzpicture}
[nodes={draw, thick, fill=black!0}]

\node (Bd2nunit) [circle, minimum width = 3.25cm] { };
\node (denote1) [draw=none, fill=none, left of = Bd2nunit, xshift= -1.4cm] {$\mathcal{B}_{d}^{(2)}(3)\ni$};
\node (disc1) [circle, minimum width = 1.1cm, left of = Bd2nunit, xshift = 0.2cm, yshift = 0.4cm] { };
\node (disc2) [circle, minimum width = 1cm, left of = Bd2nunit, xshift = 0.6cm, yshift =-0.9cm] { };
\node (disc3) [circle, minimum width = 1.3cm, right of = Bd2nunit, xshift =-0.2cm, yshift = 0.2cm] { };
\node (1) [draw = none, fill=none, left of = Bd2nunit, xshift = 0.2cm, yshift = 0.4cm] {3};
\node (2) [draw = none, fill=none, left of = Bd2nunit, xshift=0.6cm, yshift = -0.9cm] {2};
\node (3) [draw=none, fill=none, right of = Bd2nunit, xshift=-0.2cm, yshift=0.2cm] {1};

\node (arrow1) [draw=none, fill=none, right of = Bd2nunit, xshift=0.9cm] {$\rightarrow$};

\node (moveintostandard) [circle, right of = arrow1, minimum width = 3.25cm, xshift=0.9cm] { };
\node (edisc4) [circle, dashed, minimum width=0.6cm, left of = moveintostandard, yshift = 0.5cm] { }; 
\node (edisc5) [circle, dashed, minimum width=0.6cm, left of = moveintostandard, xshift=0.9cm, yshift = 0.5cm] { };
\node (edisc6) [circle, dashed, minimum width=0.6cm, left of = moveintostandard, xshift=1.8cm, yshift = 0.5cm] { };
\node (edisc7) [circle, dashed, minimum width=1.5cm, left of = moveintostandard, xshift=1cm, yshift=-0.75cm] { };
\node (e11) [draw=none, fill=none, left of = moveintostandard, yshift=0.5cm] {$3'$};
\node (e22) [draw=none, fill=none, left of = moveintostandard, xshift=0.9cm, yshift=0.5cm] {$2'$};
\node (e33) [draw=none, fill=none, left of = moveintostandard, xshift=1.8cm, yshift=0.5cm] {$1'$};
\node (edisc8) [circle, minimum width = 0.507cm, left of = edisc7, xshift = 0.7cm, yshift = 0.2cm] { };
\node (edisc9) [circle, minimum width = 0.461cm, left of = edisc7, xshift = 0.9cm, yshift = -0.4cm] { };
\node (edisc10) [circle, minimum width = 0.6cm, left of =edisc7, xshift = 1.36cm, yshift = 0.1cm] { };
\node (e111) [draw=none, fill=none, left of = edisc7, xshift = 0.7cm, yshift=0.2cm] {3};
\node (e222) [draw = none, fill=none, left of = edisc7, xshift = 0.9cm, yshift = -0.4cm] {2};
\node (e333) [draw=none, fill=none, left of =edisc7, xshift = 1.36cm, yshift = 0.1cm] {1};

\node (extraarrow) [draw=none, fill=none, right of = moveintostandard, xshift=0.9cm] {$\rightarrow$};

\node (Bd2nnunit) [circle, right of = extraarrow, minimum width = 3.25cm, xshift=0.9cm] { };
\node (disc4) [circle, dashed, minimum width=0.6cm, left of = Bd2nnunit, yshift = 0.5cm] { }; 
\node (disc5) [circle, dashed, minimum width=0.6cm, left of = Bd2nnunit, xshift=0.9cm, yshift = 0.5cm] { };
\node (disc6) [circle, dashed, minimum width=0.6cm, left of = Bd2nnunit, xshift=1.8cm, yshift = 0.5cm] { };
\node (disc7) [circle, dashed, minimum width=1.5cm, left of = Bd2nnunit, xshift=1cm, yshift=-0.75cm] { };
\node (11) [draw=none, fill=none, left of = Bd2nnunit, yshift=0.5cm] {$3'$};
\node (22) [draw=none, fill=none, left of = Bd2nnunit, xshift=0.9cm, yshift=0.5cm] {$2'$};
\node (33) [draw=none, fill=none, left of = Bd2nnunit, xshift=1.8cm, yshift=0.5cm] {$1'$};
\node (disc8) [circle, minimum width = 0.507cm, left of = disc7, xshift = 0.7cm, yshift = 0.2cm] { };
\node (disc9) [circle, minimum width = 0.461cm, left of = disc7, xshift = 0.9cm, yshift = -0.4cm] { };
\node (disc10) [circle, minimum width = 0.6cm, left of =disc7, xshift = 1.36cm, yshift = 0.1cm] { };
\node (111) [draw=none, fill=none, left of = disc7, xshift = 0.7cm, yshift=0.2cm] {3};
\node (222) [draw = none, fill=none, left of = disc7, xshift = 0.9cm, yshift = -0.4cm] {2};
\node (333) [draw=none, fill=none, left of =disc7, xshift = 1.36cm, yshift = 0.1cm] {1};

\node (arrow2) [draw=none, fill=none, below of = Bd2nunit, xshift=-2cm, yshift=-2.5cm] {$\rightarrow$};

\node (Bd3nunit) [circle, right of = arrow2, minimum width = 3.25cm, xshift=0.9cm] { };
\node (disc11) [circle, dashed, minimum width=0.6cm, left of = Bd3nunit, yshift = 0.5cm] { }; 
\node (disc12) [circle, dashed, minimum width=0.6cm, left of = Bd3nunit, xshift=0.9cm, yshift = 0.5cm] { };
\node (disc13) [circle, minimum width=0.6cm, left of = Bd3nunit, xshift=1.8cm, yshift = 0.5cm] { };
\node (1111) [draw=none, fill=none, left of = Bd3nunit, yshift=0.5cm] {$3'$};
\node (2222) [draw=none, fill=none, left of = Bd3nunit, xshift=0.9cm, yshift=0.5cm] {$2'$};
\node (3333) [draw=none, fill=none, left of = Bd3nunit, xshift=1.8cm, yshift=0.5cm] {1};
\node (disc14) [circle, dashed, minimum width=1.5cm, left of = Bd3nunit, xshift=1cm, yshift=-0.75cm] { };
\node (disc15) [circle, minimum width = 0.507cm, left of = disc14, xshift = 0.7cm, yshift = 0.2cm] { };
\node (disc16) [circle, minimum width = 0.461cm, left of = disc14, xshift = 0.9cm, yshift = -0.4cm] { };
\node (11111) [draw=none, fill=none, left of = disc14, xshift = 0.7cm, yshift=0.2cm] {3};
\node (22222) [draw = none, fill=none, left of = disc14, xshift = 0.9cm, yshift = -0.4cm] {2};

\node (Bd3nunit2) [circle, below of = moveintostandard, minimum width = 3.25cm, yshift=-2.5cm] { };
\node (disc11) [circle, dashed, minimum width=0.6cm, left of = Bd3nunit2, yshift = 0.5cm] { }; 
\node (disc12) [circle, minimum width=0.6cm, left of = Bd3nunit2, xshift=0.9cm, yshift = 0.5cm] { };
\node (disc13) [circle, minimum width=0.6cm, left of = Bd3nunit2, xshift=1.8cm, yshift = 0.5cm] { };
\node (11111) [draw=none, fill=none, left of = Bd3nunit2, yshift=0.5cm] {$3'$};
\node (22222) [draw=none, fill=none, left of = Bd3nunit2, xshift=0.9cm, yshift=0.5cm] {2};
\node (33333) [draw=none, fill=none, left of = Bd3nunit2, xshift=1.8cm, yshift=0.5cm] {1};
\node (disc17) [circle, dashed, minimum width=1.5cm, left of = Bd3nunit2, xshift=1cm, yshift=-0.75cm] { };
\node (disc18) [circle, minimum width = 0.507cm, left of = disc17, xshift = 0.7cm, yshift = 0.2cm] { };
\node (111111) [draw=none, fill=none, left of = disc17, xshift = 0.7cm, yshift=0.2cm] {3};

\node (arrow5) [draw=none, fill=none, right of = Bd3nunit2, xshift=1cm] {$\rightarrow$};

\node (Bd3nunit3) [circle, right of = arrow5, minimum width = 3.25cm, xshift=1cm] { };
\node (disc11) [circle, minimum width=0.6cm, left of = Bd3nunit3, yshift = 0.5cm] { }; 
\node (disc12) [circle, minimum width=0.6cm, left of = Bd3nunit3, xshift=0.9cm, yshift = 0.5cm] { };
\node (disc13) [circle, minimum width=0.6cm, left of = Bd3nunit3, xshift=1.8cm, yshift = 0.5cm] { };
\node (1111111) [draw=none, fill=none, left of = Bd3nunit3, yshift=0.5cm] {3};
\node (2222222) [draw=none, fill=none, left of = Bd3nunit3, xshift=0.9cm, yshift=0.5cm] {2};
\node (3333333) [draw=none, fill=none, left of = Bd3nunit3, xshift=1.8cm, yshift=0.5cm] {1};

\node (denote2) [draw=none, fill=none, right of = Bd3nunit3, xshift=1.4cm] {$\in \mathcal{B}_{d}^{(3)}(3)$};
\node (arrow4) [draw=none, fill=none, left of=Bd3nunit2, xshift=-1cm] {$\rightarrow$};

\draw (3.75,-4) -- (3.05, -2.8);
\draw (3.35,-4.25) -- (2.55,-3.15);
\draw [dashed, ->] (3.4, -3.85) -- (3, -3.2);
\draw (0.03,-4.73) -- (0.1,-3);
\draw (-0.43,-4.7) -- (-0.5,-3);
\draw [dashed, ->] (-0.2,-4.4) -- (-0.2,-3.3);
\draw (8.2,-0.85) -- (8.7,0.45);
\draw (7.7,-.5) -- (8.1,0.6);
\draw [dashed, ->] (8.1,-0.35) -- (8.3,0.2);
\end{tikzpicture}
\end{figure}
\vspace{0.3cm}
\noindent Notice that when disc 2 is moved to $2'$, it overlaps with disc 3. However since the movement occurs in $B_{d}^{(3)}(n)$, it does not cause any issues.

	\end{proof}

As an immediate consequence of Lemma \ref{lemmaoverlappingdiscinclusionnullhom}, we get the following corollary:

\begin{corollary}\label{corollary}
	For $d\geq1$, the sequence of inclusions $\mathcal{B}_{d}^{(2)}\subset\mathcal{B}_{d}^{(3)}\subset\mathcal{B}_{d}^{(4)}\subset\dots$ induces maps in the homology with each map factoring through $\mathcal{C}$om.
	
	\begin{center}
		\begin{tikzpicture}
		[nodes={fill=black!0}]
		\node(hb2) [ ] {$H_{\ast}\mathcal{B}_{d}^{(2)}$};
		\node (hb3) [right of=hb2, xshift=1cm] {$H_{\ast}\mathcal{B}_{d}^{(3)}$};
		\node (hb4) [right of=hb3, xshift=1cm] {$H_{\ast}\mathcal{B}_{d}^{(4)}$};
		\node (empty) [right of=hb4, xshift=0.5cm] { };
		\node (com1) [below of=hb2, xshift=1cm] {$\mathcal{C}$om};
		\node (com2) [below of=hb3, xshift=1cm] {$\mathcal{C}$om};
		\node (com3) [below of=hb4, xshift=1cm] {$\mathcal{C}$om};
		\node (dots1) [right of=empty, xshift=-0.5cm] {$\cdots$};
		\node (dots2) [below of=dots1] {$\cdots$};
		
		\draw [->] (hb2) -- (hb3);
		\draw [->] (hb3) -- (hb4);
		\draw [->] (hb4) -- (empty);
		\draw [->] (hb2) -- (com1);
		\draw [right hook->] (com1) -- (hb3);
		\draw [->] (hb3) -- (com2);
		\draw [right hook->] (com2) -- (hb4);
		\draw [->] (hb4) -- (com3);
		\draw [right hook->] (com3) -- (empty);
		\end{tikzpicture}
	\end{center}
\end{corollary}
	
For $d\geq2$ or $k\geq3$, the map $H_{0}\mathcal{B}_{d}^{(k)}\rightarrow \mathcal{C}om$ is just the projection to $H_{0}\mathcal{B}_{d}^{(k)}=\mathcal{C}om$. For $d=1, k=2$, $H_{\ast}\mathcal{B}_{1}^{(2)}=\mathcal{A}ssoc$. The map $H_{\ast}\mathcal{B}_{1}^{(2)}\rightarrow \mathcal{C}om$ is the natural projection $\mathcal{A}ssoc\rightarrow \mathcal{C}om$. The map $\mathcal{C}om \hookrightarrow H_{\ast}\mathcal{B}_{d}^{(k)}$ is always the inclusion $H_{0}\mathcal{B}_{d}^{(k)} \rightarrow H_{\ast}\mathcal{B}_{d}^{(k)}$, as $H_{0}\mathcal{B}_{d}^{(k)}=\mathcal{C}om$ for $k\geq3$.

\section{Compositions in $H_{\ast}\mathcal{B}_{d}^{(\bullet)}$}

In the previous Section \ref{sectionfilteredoperad} we understood the sequence of maps \eqref{inducedsequencehom}, now we want to understand the compositions maps \eqref{filteredoperadhomcomp}. Corollary \ref{corollary} tells us that the map $H_{\ast}\mathcal{B}_{d}^{(k)}\rightarrow H_{\ast}\mathcal{B}_{d}^{(k+1)}$ can be factored through $\mathcal{C}$om. The spherical cycle $\{x_{1},...,x_{k}\}$ is the boundary of the chain (disc) $c(x_{1},...,x_{k}): \{x_{1},...,x_{k}\} = \partial (c(x_{1},...,x_{k}))$ in  $\mathcal{B}_{d}^{(k+1)}(k)$. Hence $\{x_{1},...,x_{k}\} = 0$ in $H_{\ast}\mathcal{B}_{d}^{(k+1)}(k)$, by Corollary \ref{corollary}. The chain $c(x_{1},...,x_{k})$ can be explictly described as follows:
\begin{equation}\label{sphererepresentation2}
|x_{1}|^{2}+|x_{2}|^{2}+...+|x_{k}|^{2} \leq \varepsilon^{2} \qquad \qquad \sum_{i=1}^{k} x_{i} = 0
\end{equation}
where $x_{i}$ represents the center of the $i$-th disc.
\subsection{Examples}\label{Subsectionexamples}
We want to explicitly describe the composition maps
\begin{equation*}
\circ_{i}: H_{\ast}F_{k_{1}}\mathcal{O}(n_{1}) \otimes H_{\ast}F_{k_{2}}\mathcal{O}(n_{2}) \rightarrow H_{\ast}F_{k_{1}+k_{2}}\mathcal{O}(n_{1}+n_{2}-1).
\end{equation*}
Before describing the general case of composition, let us examine some examples of the composition maps on the level of homology.

Let $k_{1}=k_{2}=3$, $n_{1}=5$, and $n_{2}=3$:
\begin{equation*}
\circ_{i}:H_{\ast}\mathcal{B}_{d}^{(3)}(5) \otimes H_{\ast}\mathcal{B}_{d}^{(3)}(3) \rightarrow H_{\ast}\mathcal{B}_{d}^{(4)}(7).
\end{equation*}

Let us examine the composition when $i=5$:
\begin{equation*}
[\{x_{1},x_{2},x_{3}\},x_{4}]\cdot x_{5} \circ_{5} \{x_{1},x_{2},x_{3}\} = [\{x_{1},x_{2},x_{3}\},x_{4}]\cdot \{x_{5},x_{6},x_{7}\}\in \mathrm{Im}(H_{\ast}\mathcal{B}_{d}^{(3)}(7))\subset H_{\ast}\mathcal{B}_{d}^{(4)}(7).
\end{equation*}

Here we insert the second brace in $x_{5}$. This element is in the image of $H_{\ast}\mathcal{B}_{d}^{(3)}(7)$. By the Corollary \ref{corollary}, $[\{x_{1},x_{2},x_{3}\},x_{4}]\cdot \{x_{5},x_{6},x_{7}\} = 0$. Explicitly this element is the boundary of the chain $[\{x_{1},x_{2},x_{3}\},x_{4}]\cdot c(x_{5},x_{6},x_{7})$.

Now let us look at the composition when $i=4$:

\begin{equation*}
[\{x_{1},x_{2},x_{3}\},x_{4}]\cdot x_{5} \circ_{4} \{x_{1},x_{2},x_{3}\} = [\{x_{1},x_{2},x_{3}\}, \{x_{4},x_{5},x_{6}\}] \cdot x_{7}.
\end{equation*}

Note that 

\begin{equation*}
[ \{x_{1},x_{2},x_{3}\}, \{x_{4},x_{5},x_{6}\}] \cdot x_{7} = \partial([\{x_{1},x_{2},x_{3}\}, c(x_{4},x_{5},x_{6})] \cdot x_{7}).
\end{equation*}

This element is also zero by the same argument as above. 

Lastly, let us look at the composition for $i=3$:

\begin{equation*}
[\{x_{1},x_{2},x_{3}\},x_{4}]\cdot x_{5} \circ_{3} \{x_{1},x_{2},x_{3}\} = [\{x_{1},x_{2},\{x_{3},x_{4},x_{5}\}\},x_{6}]\cdot x_{7}.
\end{equation*}

Note that compositions $\circ_{1}$ and $\circ_{2}$ give similar results to $\circ_{3}$. We claim that resulting element $[\{x_{1},x_{2},\{x_{3},x_{4},x_{5}\}\},x_{6}]\cdot x_{7}\in H_{\ast}\mathcal{B}_{d}^{(4)}(7)$ is non-trivial. To show this, we must understand the composition of braces, that is, $\{x_{1},x_{2},\{x_{3},x_{4},x_{5}\}\}$, which can be realized as a map $S^{2d-1}\times S^{2d-1}\rightarrow \mathcal{B}_{d}^{(4)}(7)$. We can geometrically see $\{x_{1},x_{2},\{x_{3},x_{4},x_{5}\}\}$ as follows:

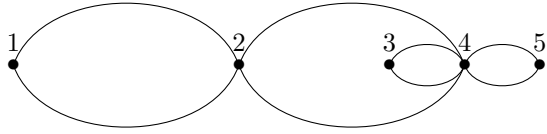
\begin{tikzpicture}

\node (center) [draw=none, fill=none] { };
\node (X1) [draw=none, fill=none, right of = center, xshift=2cm] {1};
\node (X2) [draw=none, fill=none, right of = X1, xshift=2cm] {2};
\node (X3) [draw=none, fill=none, right of = X2, xshift=1cm] {3};
\node (X4) [draw=none, fill=none, right of = X3] {4};
\node (X5) [draw=none, fill=none, right of = X4] {5};

\node (dot1) [draw=none, fill=none, below of = X1, yshift=0.7cm] {$\bullet$};
\node (dot2) [draw=none, fill=none, below of = X2, yshift=0.7cm] {$\bullet$};
\node (dot3) [draw=none, fill=none, below of = X3, yshift=0.7cm] {$\bullet$};
\node (dot4) [draw=none, fill=none, below of = X4, yshift=0.7cm] {$\bullet$};
\node (dot5) [draw=none, fill=none, below of = X5, yshift=0.7cm] {$\bullet$};

\draw (3,-0.3) to [out=70, in=110] (6,-0.3);
\draw (3,-0.3) to [out=-70, in=-110] (6,-0.3);
\draw (6,-0.3) to [out=70, in=110] (9,-0.3);
\draw (6,-0.3) to [out=-70, in=-110] (9,-0.3);
\draw (8,-0.3) to [out=70, in=110] (9,-0.3);
\draw (8,-0.3) to [out=-70, in=-110] (9,-0.3);
\draw (9,-0.3) to [out=70, in=110] (10,-0.3);
\draw (9,-0.3) to [out=-70, in=-110] (10,-0.3);
\end{tikzpicture}

Here $x_{3}, x_{4}, x_{5}$ orbit closely around each other while at the same time they collectively as a cluster closely orbit with $x_{1}$ and $x_{2}$.

\subsection{Sign Conventions}

We choose our sign convention so that it agrees with that of the Dobrinskaya-Turchin paper \cite{D_T}. When we describe elements in the homology of ${B}_{d}^{(k)}$ in terms of products of iterated brackets, cycles are realized as products of chains (usually spheres). We read the ordering of the factors (spheres) from the way the product of iterated brackets is written. We follow the rule that the spherical cycle corresponding to the brace is taken into account in the ordering by the left brace. Similarly, the spherical cycle corresponding  to the bracket is taken into account by the comma. Non-spherical cycles are represented by letters (such as Y and Z) and are taken into account in the ordering when they appear.

\textbf{Example:}

$\textbf{\{}x_{1},...,x_{k-1},[\textbf{Y,Z}]\}$ corresponds to a cycle realized by the product $S^{(k-1)d-1} \times Y \times S^{d-1}\times Z$. First we get $S^{(k-1)d-1}$ due to the brace since the sign contribution is placed on the left most brace. Then we have all of the $x_{i}$'s which do not contribute to the sign. Next we have the contribution from $Y$ and then $S^{d-1}$, the contribution from the bracket, recall that it is taken into account by the comma. Lastly, we have the contribution from $Z$. 

\subsection{Compositions}

From the examples in Section \ref{Subsectionexamples} one can see that many compositions will be trivial. In fact, all non-trivial elements come from either composing braces inside of braces, or from degree 0 classes. We can categorize elements in $H_{\ast}\mathcal{B}_{d}^{(k)}$ into 3 different types:

\noindent I) $H_{0}\mathcal{B}_{d}^{(k)}(n)$ = $\mathcal{C}om(n)$.\newline
II) Products with exactly one iterated bracket (the other factors singletons) that contain exactly one brace.\newline
III) The space spanned by all the other products of iterated brackets. In particular, elements of this type will have at least two braces.

For all $i$, I $\circ_{i}$ I $\neq 0$ as these are just compositions in $\mathcal{C}om$. Compositions of the form II $\circ_{i}$ II $\neq 0$ if and only if $i$ is inside of the brace. We claim all the other compositions are trivial.

\begin{proposition} \label{lemmacompositionstrivial}
	Assume both $k_{1}$ and $k_{2}$ are greater than 2. Then the composition maps
	\begin{equation}
	\circ_{i}: H_{\ast}\mathcal{B}_{d}^{(k_{1})}(n_{1}) \otimes H_{\ast}\mathcal{B}_{d}^{(k_{2})}(n_{2}) \rightarrow H_{\ast}\mathcal{B}_{d}^{(k_{1}+k_{2}-2)}(n_{1}+n_{2}-1)
	\end{equation}
	are trivial restricted on I $\circ_{i}$ II, I $\circ_{i}$ III, II $\circ_{i}$ I, II $\circ_{i}$ III, III $\circ_{i}$ I, III $\circ_{i}$ II, III $\circ_{i}$ III for all $i$ and on II $\circ_{i}$ II if $i$ is outside of the brace.
\end{proposition}

\begin{proof}
	\noindent \textit{Cases $I\circ_{i}II$} and \textit{$I\circ_{i}III$}: 
	Let $x_{1}\cdot...\cdot x_{n_{1}}\in H_{0}\mathcal{B}_{d}^{(k_{1})}(n_{1})$ be of type I and $\beta(x_{1},...,x_{n_{2}}) \in H_{\ast}\mathcal{B}_{d}^{(k_{2})}(n_{2})$ be of type II or III.
	Compositions in these cases are as follows:
	\begin{equation*}
	x_{1}\cdot...\cdot x_{n_{1}}\circ_{i} \beta(x_{1},...,x_{n_{2}})= x_{1}\cdot...\cdot x_{i-1}\cdot\beta(x_{i},...,x_{i+n_{2}-1})\cdot...\cdot x_{n_{1}+n_{2}-1}.
	\end{equation*}
	
	The result of this composition is in $\mathrm{Im}(H_{>0}\mathcal{B}_{d}^{(k_{2})}(n_{1}+n_{2}-1)) \subset H_{>0}\mathcal{B}_{d}^{(k_{1}+k_{2}-2)}(n_{1}+n_{2}-1)$. Since $k_{1}$ is assumed to be greater than 2, by Lemma \ref{lemmaoverlappingdiscinclusionnullhom}, the compositions of these forms are trivial.
	
	\textit{Cases $II\circ_{i}I$} and \textit{$III\circ_{i}I$}: 
	Let $\alpha(x_{1},...,x_{n_{1}}) \in H_{\ast}\mathcal{B}_{d}^{(k_{1})}(n_{1})$ be of type II or III and $x_{1}\cdot...\cdot x_{n_{2}}\in H_{0}\mathcal{B}_{d}^{(k_{2})}(n_{2})$ be of type I. Compositions in these cases are as follows:
	\begin{equation*}
	\alpha(x_{1},...,x_{n_{1}}) \circ_{i} x_{1}\cdot...\cdot x_{n_{1}} = \alpha(x_{1},...,x_{i-1}, x_{i}\cdot...\cdot x_{i+n_{2}-1},x_{i+n_{2}},...,x_{n_{1}+n_{2}-1}).
	\end{equation*}
	If $i$ is not in the brace of $\alpha(x_{1},...,x_{n_{1}})$ then the result of the composition is clearly in $\mathrm{Im}(H_{>0}\mathcal{B}_{d}^{(k_{1})}(n_{1}+n_{2}-1))\subset H_{>0}\mathcal{B}_{d}^{(k_{1}+k_{2}-2)}(n_{1}+n_{2}-1)$. Since $k_{2}$ is assumed to be greater than 2, by Lemma \ref{lemmaoverlappingdiscinclusionnullhom}, the compositions of these forms are trivial.
	
	\textit{Case $II\circ_{i}II$}: 
	Let $\alpha(x_{1},...,x_{n_{1}})\in H_{\ast}\mathcal{B}_{d}^{(k_{1})}(n_{1})$ be an element of type II and let $\beta(x_{1},...,x_{n_{2}})\in H_{\ast}\mathcal{B}_{d}^{(k_{2})}(n_{2})$ be another element of type II.
	
	Compositions are as follows:
	\begin{equation*}
	\alpha(x_{1},...,x_{n_{1}}) \circ_{i} \beta(x_{1},...,x_{n_{2}}) = \alpha(x_{1},...,x_{i-1},\beta(x_{i},..,x_{i+n_{2}-1
	}),x_{i+n_{2}},...,x_{n_{1}+n_{2}-1}).
	\end{equation*}
	Since $i$ is not inside the brace, the result of the composition is in $\mathrm{Im}(H_{>0}\mathcal{B}_{d}^{(k')}(n_{1}+n_{2}-1))$, where $k' = \mathrm{max}(k_{1},k_{2})$. Then  $\mathrm{Im}(H_{>0}\mathcal{B}_{d}^{(k')}(n_{1}+n_{2}-1))\subset H_{>0}\mathcal{B}_{d}^{(k_{1}+k_{2}-2)}(n_{1}+n_{2}-1)$. Since $k_{1}$ and $k_{2}$ are assumed to be greater than 2, by Lemma \ref{lemmaoverlappingdiscinclusionnullhom}, any composition of this form is trivial.
	
	\textit{Cases $III\circ_{i}II$} and \textit{$III\circ_{i}III$}: 	
	Let $\alpha(x_{1},...,x_{n_{1}})\in H_{\ast}\mathcal{B}_{d}^{(k_{1})}(n_{1})$ be an element of type III and let $\beta(x_{1},...,x_{n_{2}})\in H_{\ast}\mathcal{B}_{d}^{(k_{2})}(n_{2})$ be an element of type II or III. Compositions are as follows:
	\begin{equation*}
	\alpha(x_{1},...,x_{n_{1}}) \circ_{i}\beta(x_{1},...,x_{n_{2}}) = \alpha(x_{1},...,x_{i-1},\beta(x_{i},...,x_{i+n_{2}-1}),...,x_{n_{1}+n_{2}-1}).
	\end{equation*}
	Since $\alpha(x_{1},...,x_{n_{1}})$ is of type III, it has at least two braces. Therefore for any $i$, there is at least one brace of $\alpha(x_{1},...,x_{n_{1}})$ that is unaffected by the composition. That brace is of length $k_{1}$ and is the boundary of the  chain $c(x_{1},...,x_{k_{1}})$. Therefore the composition is trivial.
	
	$\mathbf{Example:}$ For $k_{1}=3$, $k_{2}=4$, $n_{1}=6$, $n_{2}=4$, we get a map $\circ_{2}: \mathcal{B}_{d}^{(3)}(6) \times \mathcal{B}_{d}^{(4)}(4) \rightarrow \mathcal{B}_{d}^{(5)}(9)$. Let $\alpha = [\{x_{1},x_{2},x_{3}\},\{x_{4},x_{5},x_{6}\}]$ and $\beta = \{x_{1},x_{2},x_{3},x_{4}\}$. The composite cycle $\alpha \circ_{2} \beta$ can be realized as a product of spheres $S^{2d-1}\times S^{3d-1}\times S^{d-1}\times S^{2d-1}$, which is a boundary $\partial (S^{2d-1}\times S^{3d-1}\times S^{d-1}\times D^{2d})$:
	\begin{align*}
	[\{x_{1},x_{2},x_{3}\},\{x_{4},x_{5},x_{6}\}]\circ_{2} \{x_{1},x_{2},x_{3},x_{4}\} &= [\{x_{1},\{x_{2},x_{3},x_{4},x_{5}\},x_{6}\},\{x_{7},x_{8},x_{9}\}]\\ &=\partial[\{x_{1},\{x_{2},x_{3},x_{4},x_{5}\},x_{6}\},c(x_{7},x_{8},x_{9})]\\ &=0.
	\end{align*}
	For the inner brace, $\{x_{2},x_{3},x_{4},x_{5}\}$, any three of the discs 2, 3, 4, and 5 can overlap. If we replace $\{x_{2},x_{3},x_{4},x_{5}\}$ with Y in the brace $\{x_{1},\{x_{2},x_{3},x_{4},x_{5}\},x_{6}\}$, we have $\{x_{1},Y,x_{6}\}$ and any two of the discs 1, $Y$, and 6 can overlap. Therefore if $Y$ overlaps with either 1 or 6 there is at most four discs overlapping, so the result is indeed in $\mathcal{B}_{d}^{(5)}(9)$. However since $\{x_{7},x_{8},x_{9}\}$ bounds a disc, as seen above, the result of the composition is zero in $\mathcal{B}_{d}^{(5)}(9)$.
	
	\medskip
	
	The remaining case $II\circ_{i}III$ is the most difficult and we will need the following lemma:
	
	\begin{lemma} \label{newleibnizeqs}
		Let $\{x_{1},...,x_{k_{1}}\}\in H_{(k_{1}-1)d-1}\mathcal{B}_{d}^{(k_{1})}(k_{1})$ and
		let $\beta(x_{1},...,x_{n_{2}})\in H_{\ast}\mathcal{B}_{d}^{(k_{2})}(n_{2})$ be an element of type III.
		
		We can either write $\beta = Y\cdot Z$ or $\beta = [Y, Z]$ for some $Y$ and $Z$.
		When $\beta = Y\cdot Z$, then
		\begin{equation}\label{lemma4.5leibnizeq1}
		\{x_{1},...,x_{k_{1}-1},(Y\cdot Z)\} = (-1)^{\abs{Y}((k_{1}-1)d-1)} Y\cdot\{x_{1},...,x_{k_{1}-1},Z\} + \{x_{1},...,x_{k_{1}-1},Y\} \cdot Z,
		\end{equation}
		
		and when $\beta = [Y, Z]$, then
		\begin{equation}\label{lemma4.5leibnizeq2}
		\{x_{1},...,x_{k_{1}-1},[Y, Z]\} = (-1)^{(\abs{Y}+d-1)((k_{1}-1)d-1)} [Y,\{x_{1},...,x_{k_{1}-1},Z\}] + [\{x_{1},...,x_{k_{1}-1},Y\},Z].
		\end{equation}
		
	\end{lemma}
	
	\begin{proof}
		The formulas \eqref{lemma4.5leibnizeq1} and \eqref{lemma4.5leibnizeq2} do not follow from \eqref{thmleibniz1} and \eqref{thmleibniz2}, but are proved by a similar argument as in \cite[Examples 5.2 and 5.3]{D_T}, which was inspired from \cite{Bary}. Consider the cycle $\{x_{1},...,x_{k_{1}-1},(Y\cdot Z)\}$.  When we pull $Z$ far away, it forms a chain,
		which might have a forbidden $(k_{1}+k_{2}-1)$-overlap. This could only happen near the plane  $$x_{1}=...=x_{k_{1}-1}=Z,$$ where abusing notation, $Z$ denotes the center of mass of points appearing in the chain $Z$. We remove a small tubular neighborhood of this forbidden plane and this produces the cycle $Y\cdot\{x_{1},...,x_{k_{1}-1},Z\}$. On the other hand when $Z$ is far away we get the cycle $\{x_{1},...,x_{k_{1}-1},Y\} \cdot Z$. This proves the relation \eqref{lemma4.5leibnizeq1}. 
		
		Geometrically we can see the chain $C$ below.
		
		
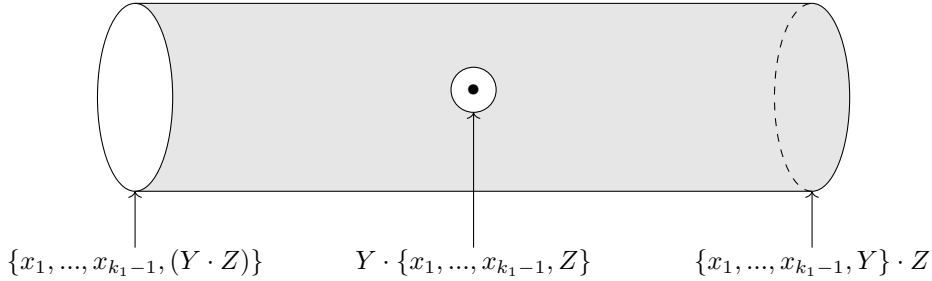
\vspace{0.3cm}
		
		\begin{tikzpicture}
		[nodes={draw, thick, fill=black!0}]
		
		\fill [gray!20] (2,-1.25) -- (2,-1.25) arc (270:90:-0.5  and 1.25) -- (11,1.25) arc (90:270:-0.5 and 1.25);
		\draw (2,0) ellipse (0.5 and 1.25);
		\draw (2,-1.25) -- (11,-1.25);
		\draw (2,1.25) -- (11,1.25);
		\draw (11,-1.25) arc (270:90:-0.5  and 1.25);
		\draw [dashed] (11,-1.25) arc (270:90:0.5  and 1.25);
		\filldraw[color=black!100, fill=black!0] (6.5,0.1) circle (0.3);
		
		\draw [->] (2,-2) -- (2,-1.25);
		\draw [->] (6.5,-2) -- (6.5,-0.2); 
		\draw [->] (11,-2) -- (11,-1.25);
		
		\node (center) [draw=none, fill=none] { };
		\node (bullet) [draw=none, fill=none, right of=center, xshift=5.5cm, yshift=0.1cm] {$\bullet$};
		\node (label1) [draw=none, fill=none, right of=center, xshift=1cm, yshift=-2.2cm] {$\{x_{1},...,x_{k_{1}-1},(Y\cdot Z)\}$};
		\node (label2) [draw=none, fill=none, right of=center, xshift=5.5cm, yshift=-2.2cm] {$Y\cdot\{x_{1},...,x_{k_{1}-1},Z\}$};
		\node (label3) [draw=none, fill=none, right of=center, xshift=10cm, yshift=-2.2cm] {$\{x_{1},...,x_{k_{1}-1},Y\} \cdot Z$};
		\end{tikzpicture}
		
		\vspace{0.3cm}
		
		Now consider the cycle $\{x_{1},...,x_{k_{1}-1},[Y, Z]\}$. We pull $[Y,Z]$ together far away. This produces a chain that intersects forbidden strata. Notice $Y$ and $Z$ rotate around one another and thus never meet.
		This chain meets the plane
		\[x_{1}=x_{2}=...=x_{k_{1}-1} = Z.\] By removing a tubular neighborhood of this intersection with the chain we get the cycle $\{x_{1},...,x_{k_{1}-1},Z\}$ at the boundary near every point of intersection. Simultaneously this cycle rotates around $Y$ since $x_{1},...,x_{k_{1}-1}$ have collided with $Z$. Hence we have $[Y,\{x_{1},...,x_{k_{1}-1},Z\}]$ as part of the boundary of our chain. Similarly the intersection with the plane  \[x_{1}=x_{2}=...=x_{k_{1}-1}=Y\] produces the cycle $[\{x_{1},...,x_{k_{1}-1},Y\},Z]$. On the other end of the cylinder, the boundary is given by $\{x_{1},...,x_{k_{1}-1},1\} \cdot [Y,Z]$ and is 0 in the homology as $\{x_{1},...,x_{k_{1}-1},1\} =0$ in $H_{>0}\mathcal{B}_{d}^{(k_{1})}(k_{1}-1)=0$. There are no other forbidden planes to contribute and thus we obtain \eqref{lemma4.5leibnizeq2}. Geometrically we can see the chain $C$ below.
		
		
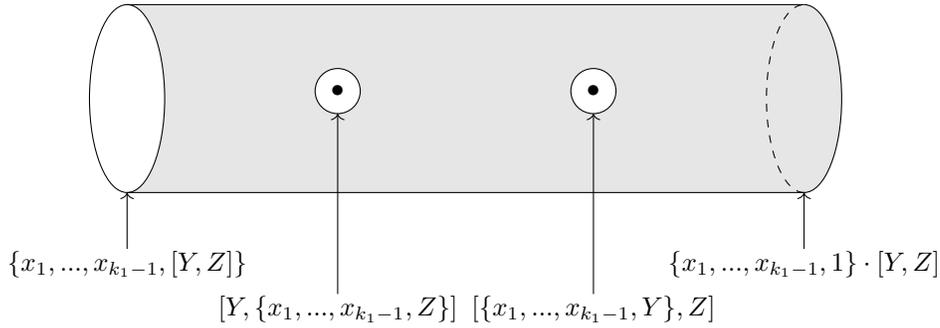
\vspace{0.3cm}
		
		\begin{tikzpicture}
		[nodes={draw, thick, fill=black!0}]
		
		\fill [gray!20] (2,-1.25) -- (2,-1.25) arc (270:90:-0.5  and 1.25) -- (11,1.25) arc (90:270:-0.5 and 1.25);
		\draw (2,0) ellipse (0.5 and 1.25);
		\draw (2,-1.25) -- (11,-1.25);
		\draw (2,1.25) -- (11,1.25);
		\draw (11,-1.25) arc (270:90:-0.5  and 1.25);
		\draw [dashed] (11,-1.25) arc (270:90:0.5  and 1.25);
		\filldraw[color=black!100, fill=black!0] (4.8,0.1) circle (0.3);
		\filldraw[color=black!100, fill=black!0] (8.2,0.1) circle (0.3);
		
		\draw [->] (2,-2) -- (2,-1.25);
		\draw [->] (4.8,-2.6) -- (4.8,-0.2); 
		\draw [->] (8.2,-2.6) -- (8.2,-0.2);
		\draw [->] (11,-2) -- (11,-1.25);
		
		\node (center) [draw=none, fill=none] { };
		\node (bullet1) [draw=none, fill=none, right of=center, xshift=3.8cm, yshift=0.1cm] {$\bullet$};
		\node (bullet2) [draw=none, fill=none, right of=center, xshift=7.2cm, yshift=0.1cm] {$\bullet$};
		\node (label1) [draw=none, fill=none, right of=center, xshift=1cm, yshift=-2.2cm] {$\{x_{1},...,x_{k_{1}-1},[Y, Z]\}$};
		\node (label2) [draw=none, fill=none, right of=center, xshift=3.8cm, yshift=-2.8cm] {$[Y,\{x_{1},...,x_{k_{1}-1},Z\}]$};
		\node (label3) [draw=none, fill=none, right of=center, xshift=7.2cm, yshift=-2.8cm] {$[\{x_{1},...,x_{k_{1}-1},Y\},Z]$};
		\node (label4) [draw=none, fill=none, right of=center, xshift=10cm, yshift=-2.2cm] {$\{x_{1},...,x_{k_{1}-1},1\} \cdot [Y,Z]$};
		\end{tikzpicture}	
		
	\end{proof}	
	
	\textit{Case $II\circ_{i}III$}: 
	Let $\alpha(x_{1},...,x_{n_{1}})\in H_{\ast}\mathcal{B}_{d}^{(k_{1})}(n_{1})$ be an element of type II and let $\beta(x_{1},...,x_{n_{2}})\in H_{\ast}\mathcal{B}_{d}^{(k_{2})}(n_{2})$ be an element of type III. Compositions are as follows:
	\begin{equation*}
	\alpha(x_{1},...,x_{n_{1}}) \circ_{i} \beta(x_{1},...,x_{n_{2}}) = \alpha(x_{1},...,x_{i-1},\beta(x_{i},..,x_{i+n_{2}-1}),x_{i+n_{2}},...,x_{n_{1}+n_{2}-1})
	\end{equation*}
	If $i$ is not in the brace of $\alpha(x_{1},...,x_{n_{1}})$ then the brace of $\alpha(x_{1},...,x_{n_{1}})$ bounds a disc, that is, $\{x_{i_{1}},...,x_{i_{k_{1}}}\} = \partial c(x_{i_{1}},...,x_{i_{k_{1}}})$, similar to the example above.
	
	If $i$ is in the brace of $\alpha(x_{1},...,x_{n_{1}})$ then we can apply Lemma \ref{newleibnizeqs}. After iteratively applying \eqref{lemma4.5leibnizeq1} and \eqref{lemma4.5leibnizeq2}, although there can be braces inside of braces, each summand in the result must also have $\{x_{i_{1}},...,x_{i_{k_{2}}}\}$ outside of the brace. Indeed, since $\beta(x_{1},...,x_{n_{2}})$ is of type III, and thus has at least two braces. Then  $\{x_{i_{1}},...,x_{i_{k_{2}}}\} = \partial c(x_{i_{1}},...,x_{i_{k_{2}}})$. Since the brace $\{x_{i_{1}},...,x_{i_{k_{2}}}\}$, is the boundary of a disc, composition is trivial.
\end{proof}

\begin{theorem}\label{maintheorem}
	The composition of braces inside of other braces has a new relation:
	
	\begin{align}
	\{x_{1},...,x_{k_{1}-1},\{x_{k_{1}},...,x_{k_{1}+k_{2}-1}\}\} &= -(-1)^{(k_{1}-1)d} \sum_{i=1}^{k_{1}-1} (-1)^{(i-1)d}[x_{i},\{x_{1},...,\hat{x}_{i},...,x_{k_{1}+k_{2}-1}\}]\label{sum1}\\ &= (-1)^{(k_{1}-1)d} \sum_{i=k_{1}}^{k_{1}+k_{2}-1} (-1)^{(i-1)d} [x_{i},\{x_{1},...,\hat{x}_{i},...,x_{k_{1}+k_{2}-1}\}]\label{sum2}
	\end{align}
	
\end{theorem}

Note that the difference of \eqref{sum1} and \eqref{sum2} is exactly the generalized Jacobi relation \eqref{thmjacobirelation}. 

\begin{proof} We can see the above relation geometrically as follows:

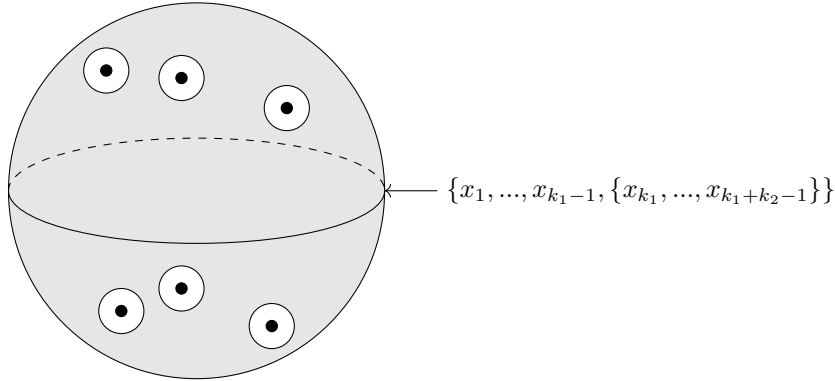
\vspace{0.3cm}
\begin{center}
\begin{tikzpicture}
\filldraw[fill=gray!20, draw=black] (2,2) circle (2.5cm);
\draw (-0.5,2) arc (-180:0:2.5 and 0.7);
\draw [dashed] (-0.5,2) arc (180:0: 2.5 and 0.7);
\filldraw[fill=white, draw=black] (3,0.2) circle (0.3cm);
\filldraw[fill=white, draw=black] (1,0.4) circle (0.3cm);
\filldraw[fill=white, draw=black] (1.8,0.7) circle (0.3cm);
\filldraw[fill=white, draw=black] (0.8,3.6) circle (0.3cm);
\filldraw[fill=white, draw=black] (3.2,3.1) circle (0.3cm);
\filldraw[fill=white, draw=black] (1.8,3.5) circle (0.3cm);
\filldraw [fill=black] (3,0.2) circle (0.075cm);
\filldraw [fill=black] (1,0.4) circle (0.075cm);
\filldraw [fill=black] (1.8,0.7) circle (0.075cm);
\filldraw [fill=black] (0.8,3.6) circle (0.075cm);
\filldraw [fill=black] (3.2,3.1) circle (0.075cm);
\filldraw [fill=black] (1.8,3.5) circle (0.075cm);
\draw [<-] (4.5,2) -- (5.2,2) node[right] {$\{x_{1},...,x_{k_{1}-1},\{x_{k_{1}},...,x_{k_{1}+k_{2}-1}\}\}$};

\end{tikzpicture}	
\end{center}
\vspace{0.3cm}

Consider the intersection of the space $\mathcal{M}_{d}^{(k_{1}+k_{2}-2)} (k_{1}+k_{2}-1)$ with the sphere given by $\sum\limits_{i=1}^{k_{1}+k_{2}-1}x_{i}=0$ and $\sum\limits_{i=1}^{k_{1}+k_{2}-1} x_{i}^{2} = 1$. This space is homotopy equivalent to $\mathcal{M}_{d}^{(k_{1}+k_{2}-2)} (k_{1}+k_{2}-1)$, where translations and rescaling have been killed. The obtained space is $S^{(k_{1}+k_{2}-2)d-1}$ with several subspheres removed. Each removed subsphere is given by the intersection of $S^{(k_{1}+k_{2}-2)d-1}$ with the plane $x_{1}=...=\hat{x}_{i}= ... =x_{k_{1}+k_{2}-1}=0$ for $1\leq i \leq k_{1}+k_{2}-1$. All the removed subspheres are disjoint. We can take tubular neighborhoods around each subsphere, which are also all disjoint. Each tubular neighborhood has boundary which is exactly the cycle $[x_{i},\{x_{1},...,\hat{x}_{i},...,x_{k_{1}+k_{2}-1}\}]$. After removing these tubular neighborhoods we are left with a manifold with boundary that we denote $X_{d}^{(k_{1}+k_{2}-2)}(k_{1}+k_{2}-1)$. The boundary of $X_{d}^{(k_{1}+k_{2}-2)}(k_{1}+k_{2}-1)$ is exactly the generalized Jacobi relation \eqref{thmjacobirelation} for $k=k_{1}+k_{2}-2$. Then $\{x_{1},...,x_{k_{1}-1},\{x_{k_{1}},...,x_{k_{1}+k_{2}-1}\}\}$ is a submanifold of $X_{d}^{(k_{1}+k_{2}-2)}(k_{1}+k_{2}-1)$ that is of codimension 1. Thus $\{x_{1},...,x_{k_{1}-1},\{x_{k_{1}},...,x_{k_{1}+k_{2}-1}\}\}$ splits $X_{d}^{(k_{1}+k_{2}-2)}(k_{1}+k_{2}-1)$ into two parts where one part is given by the right hand side of \eqref{sum1} and the other is given by \eqref{sum2}.

The cycle $\{x_{1},...,x_{k_{1}-1},\{x_{k_{1}},...,x_{k_{1}+k_{2}-1}\}\}$ can be realized as a product of spheres: $S^{(k_{1}-1)d-1}\times S^{(k_{2}-1)d-1}$. We can describe the first sphere $S^{(k_{1}-1)d-1}$ by the following equations:

\begin{equation}\label{sphere1equs}
x_{1}+\cdots +x_{k_{1}-1}+Y=0 \qquad\qquad x_{1}^{2}+\cdots +x_{k_{1}-1}^{2}+Y^{2}=c^{2}\cdot k_{1}(k_{1}-1)
\end{equation}
where $Y=\frac{1}{k_{1}-1}(x_{k_{1}+1}+...+x_{k_{1}+k_{2}-1})$. 

Define $\overline{x}_{k_{1}} = x_{k_{1}} - Y$, $\overline{x}_{k_{1}+1} = x_{k_{1}+1} - Y$,..., $\overline{x}_{k_{1}+k_{2}-1} = x_{k_{1}+k_{2}-1} - Y$. Then we can describe the second sphere by the following equations: 

\begin{equation}\label{sphere2equs}
\overline{x}_{k_{1}}+\cdots +\overline{x}_{k_{1}+k_{2}-1} = 0 \qquad\qquad \overline{x}_{k_{1}}^{2}+\cdots +\overline{x}_{k_{1}+k_{2}-1}^{2} = \epsilon^{2} \cdot k_{2}(k_{2}-1), \qquad \epsilon^2 << c^2.
\end{equation}

Next we define a chain one dimension bigger by pulling only one point, $x_{k_{1}}$, in the direction of $(1,0,...,0)$. As we pull $x_{k_{1}}$ it can collide with $x_{k_{1}+1},...,x_{k_{1}+k_{2}-1}$ only when it intersects with $(k_{1}-1)$ forbidden strata each given by the following set of equations:

\begin{equation}\label{forbiddenstrata}
x_{1} = \cdots = \hat{x}_{i} = \cdots = x_{k_{1}+k_{2}-1}, \qquad 1\leq i \leq k_{1}-1.
\end{equation}

Note that $x_{j}=x_{i}$ if and only if $\overline{x}_{j} = \overline{x}_{i}$. The obtained chain is a cylinder $S^{(k_{1}-1)d-1}\times S^{(k_{2}-1)d-1}\times [0,N]$, where $N>>0$. This cylinder intersects the $(k_{1}-1)$-forbidden strata \eqref{forbiddenstrata} transversely and disjointly. 

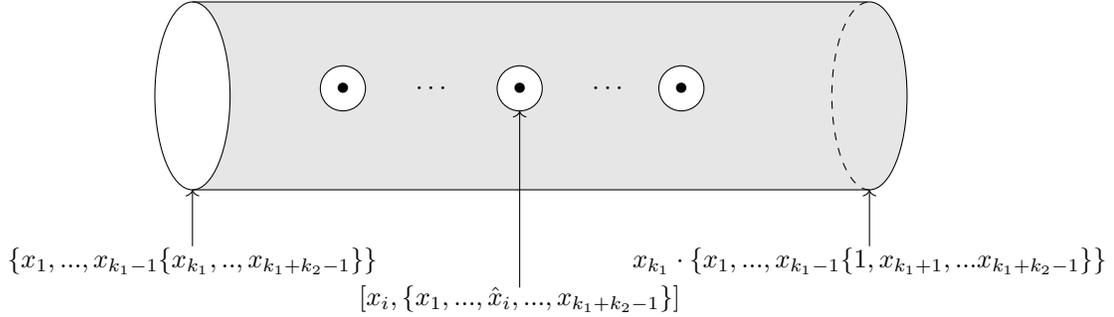
\vspace{0.3cm}

\begin{tikzpicture}
[nodes={draw, thick, fill=black!0}]

\fill [gray!20] (2,-1.25) -- (2,-1.25) arc (270:90:-0.5  and 1.25) -- (11,1.25) arc (90:270:-0.5 and 1.25);
\draw (2,0) ellipse (0.5 and 1.25);
\draw (2,-1.25) -- (11,-1.25);
\draw (2,1.25) -- (11,1.25);
\draw (11,-1.25) arc (270:90:-0.5  and 1.25);
\draw [dashed] (11,-1.25) arc (270:90:0.5  and 1.25);
\filldraw[color=black!100, fill=black!0] (4,0.1) circle (0.3);
\filldraw[color=black!100, fill=black!0] (8.5,0.1) circle (0.3);
\filldraw[color=black!100, fill=black!0] (6.35,0.1) circle (0.3);

\draw [->] (2,-2) -- (2,-1.25);
\draw [->] (11,-2) -- (11,-1.25);
\draw [->] (6.35,-2.55) -- (6.35, -0.2);

\node (center) [draw=none, fill=none] { };
\node (bullet1) [draw=none, fill=none, right of=center, xshift=3cm, yshift=0.1cm] {$\bullet$};
\node (bullet3) [draw=none, fill=none, right of=center, xshift=5.35cm, yshift=0.1cm] {$\bullet$};
\node (bullet2) [draw=none, fill=none, right of=center, xshift=7.5cm, yshift=0.1cm] {$\bullet$};
\node (dots1) [draw=none, fill=none, right of=bullet1, xshift=0.2cm] {$\cdots$};
\node (dots2) [draw=none, fill=none, right of=bullet3, xshift=0.2cm] {$\cdots$};
\node (label1) [draw=none, fill=none, right of=center, xshift=1cm, yshift=-2.2cm] {$\{x_{1},...,x_{k_{1}-1}\{x_{k_{1}},..,x_{k_{1}+k_{2}-1}\}\}$};
\node (label2) [draw=none, fill=none, right of=center, xshift=3.8cm, yshift=-2.8cm] { };
\node (label3) [draw=none, fill=none, right of=center, xshift=7.2cm, yshift=-2.8cm] { };
\node (label4) [draw=none, fill=none, right of=center, xshift=10cm, yshift=-2.2cm] {$x_{k_{1}}\cdot \{x_{1},...,x_{k_{1}-1}\{1,x_{k_{1}+1},...x_{k_{1}+k_{2}-1}\}\}$};
\node (label15) [draw=none, fill=none, below of=bullet3, yshift=-1.8cm] {$[x_{i},\{x_{1},...,\hat{x}_{i},...,x_{k_{1}+k_{2}-1}\}]$};
\end{tikzpicture}

\vspace{0.3cm}

To get an actual chain in $\mathcal{M}_{d}^{(k_{1}+k_{2}-2)} (k_{1}+k_{2}-1)$, we remove disjoint tubular neighborhoods of each intersection with the forbidden strata. 

As an example, consider the intersection with the stratum $x_{2} = \cdots = x_{k_{1}+k_{2}-1}$. It happens when the initial position of $\overline{x}_{k_{1}} = (-(k_{1}-1)\epsilon,0,...,0)$ and the position of $\overline{x}_{j} = (\epsilon, 0,...,0)$ for $j= k_{1}+1,...,k_{1}+k_{2}-1$. Indeed, when $x_{k_{1}}$ is pulled in the direction of (1,0,...,0), the same happens with $\overline{x}_{k_{1}}$ and it hits all the other $\overline{x}_{j}$, $j=k_{1}+1,...,k_{1}+k_{2}-1$ only if $\overline{x}_{k_{1}+1} = \cdots = \overline{x}_{k_{1}+k_{2}-1}$ and are in the position $(\epsilon,0,..,0)$. Now since $Y=\frac{1}{k_{1}-1}(x_{k_{1}+1}+...+x_{k_{1}+k_{2}-1})$ and all $x_{j}$, $j=k_{1}+1,...,k_{1}+k_{2}-1$ are equal, we get $Y = x_{j}$, $j=k_{1}+1,...,k_{1}+k_{2}-1$. We also need $Y$ to coincide with $x_{2},...,x_{k_{1}-1}$, this reduces \eqref{sphere1equs} to the following equations:

\begin{align}
x_{1}+ (k_{1}-1)Y &=0\label{collisionequ1}\\
x_{1}^{2}+(k_{1}-1)Y^{2} &=c^{2}k_{1}(k_{1}-1).\label{collisionequ2}
\end{align}

The first equation \eqref{collisionequ1} kills translations and the second equation \eqref{collisionequ2} kills rescaling. Thus we get a sphere $S^{d-1}$, which corresponds to $[x_{1}, Y]$. However since $Y$ has collided with $x_{2}, ..., x_{k_{1}+k_{2}-1}$, we remove from the attained chain a tubular neighborhood of its intersection with the forbidden stratum $x_{2}=\cdots = x_{k_{1}+k_{2}-1}$. This gives us exactly the cycle $[x_{1}, \{x_{2},...,x_{k_{1}+k_{2}-1}\}]$.

Therefore, in general, if $Y$ has collided with $x_{1},...,\hat{x}_{i},..,x_{k_{1}+k_{2}-1}$, for $i=i,...,k_{1-1}$, then we get similarly

\begin{align*}
x_{i}+ (k_{1}-1)Y &=0\\
x_{i}^{2}+(k_{1}-1)Y^{2} &=c^{2}k_{1}(k_{1}-1).
\end{align*}

This gives a sphere $S^{d-1}$ corresponding to $[x_{i},Y]$. We remove from the cylinder $S^{(k_{1}-1)d-1}\times S^{(k_{2}-1)d-1}\times [0,N]$ a tubular neighborhood of its intersection with the forbidden strata $x_{1}=\cdots =\hat{x}_{i}=\cdots =x_{k_{1}+k_{2}-1}$, which yields the boundary cycle $[x_{i}, \{x_{1},...,\hat{x}_{i},..,x_{k_{1}+k_{2}-1}\}]$. 

At the right end of the cylinder, $S^{(k_{1}-1)d-1}\times S^{(k_{2}-1)d-1}\times [0,N]$, we get a cycle $x_{k_{1}}\cdot \{x_{1},...,x_{k_{1}-1}\{1,x_{k_{1}+1},...x_{k_{1}+k_{2}-1}\}\}$ which is homologously trivial since $\{1,x_{1},...x_{k_{2}-1}\} = 0 \in H_{\geq0}\mathcal{M}_{d}^{(k_{2})}(k_{2}-1)=0$. All together, this gives us the relation \eqref{sum1}. Note that \eqref{sum2} minus \eqref{sum1} is the generalized Jacobi \eqref{thmjacobirelation} and therefore \eqref{sum2} is a consequence of \eqref{sum1}. 

We explain how the sign in front of the sum \eqref{sum2} is found in the next section. It is enough to understand the sign in front of only one of the summands in \eqref{sum2} (We do it for the very last one). This is due to the argument made at the beginning of the proof.

\end{proof}

\section{Signs in Theorem \ref{maintheorem}}

In \cite[Section 6]{D_T}, the authors describe the cohomology groups $H^{\ast}\mathcal{M}_{d}^{(k)}(n)$ as spaces of certain admissible $k$-forests, where the $k$-forests have two types of vertices square and round. If a $k$-forest has only one component, we call it a $k$-tree. Every square vertex contains a $(k-1)$-elements subset of $\{1,...,n\}$ and every round vertex contains just one element. Each round vertex must be connected by an edge to a single square vertex, or completely disconnected from all other vertices. Square vertices must be connected to at least one round vertex. One orients all the given edges between the vertices. Each $k$-forest has an orientation set, which consists of all the edges and square vertices. The order of the orientation set encodes the coorientations of the corresponding chains. The degree of a square vertex is $(k-2)d$ and the degree of an edge is $d-1$. The cocycles in $H^{\ast}\mathcal{M}_{d}^{(k)}(n)$ corresponding to the $k$-forests are geometrically realized as an intersection number with cooriented chains in $\mathcal{M}_{d}^{(k)}(n)$, which are defined by a set of (in)equalities as follows. If $i$ and $j$ are in the same square vertex, then $x_{i} = x_{j}$. The authors give a projection $p_{1}:\mathbb{R}^{d}\rightarrow \mathbb{R}^{d-1}$ where $(x^{1},...,x^{d})\mapsto (x^{2},...,x^{d})$. For two vertices $A$ and $B$ in a forest that are connected by an edge oriented from $A$ to $B$, they require that $B$ is ``above" $A$. Explicitly, for all $i\in A$ and all $j\in B$, $x_{i}^{1}\leq x_{j}^{1}$ and $p_{1}(x_{i}) = p_{1}(x_{j})$.

In \cite[Section 8]{D_T}, the authors define a map $\Psi$, which describes the intersection between cycles given as products of iterated brackets (and geometrically realized as products of spheres) with the $k$-forests cocycles.

As an example,
\begin{equation}\label{psi}
\Psi(\{x_{1},...,x_{k}\}) = \sum\limits_{\ell = 1}^{k} (-1)^{(\ell-1)d} \begin{tikzpicture} [baseline=-0.7cm]
\draw node [left, yshift=0.2cm] {1} (0,0) rectangle node {$1,...,\hat{\ell},...,k$} (2.1,0.5); 
\draw (1.05,-1.5) circle (0.3) node {$\ell$};
\draw [->] (1.05,0) -- (1.05,-1.2) node [left, yshift=0.5cm] {2};
\end{tikzpicture}
\end{equation}
The formula \eqref{psi} means that the intersection of the cycle $\{x_{1},...,x_{k}\}$ and the cocycle \begin{tikzpicture} [baseline=-0.4cm]
\draw node [left, yshift=0.2cm] {1} (0,0) rectangle node {$1,...,\hat{\ell},...,k$} (2.1,0.5); 
\draw (1.05,-1) circle (0.3) node {$\ell$};
\draw [->] (1.05,0) -- (1.05,-0.7) node [left, yshift=0.3cm] {2};
\end{tikzpicture} is $(-1)^{(\ell-1)d}$, i.e. $\{x_{1},...,x_{k}\} \; \bigcap$ \begin{tikzpicture} [baseline=-0.4cm]
\draw node [left, yshift=0.2cm] {1} (0,0) rectangle node {$1,...,\hat{\ell},...,k$} (2.2,0.5); 
\draw (1.05,-1) circle (0.3) node {$\ell$};
\draw [->] (1.05,0) -- (1.05,-0.7) node [left, yshift=0.3cm] {2};
\end{tikzpicture} = $(-1)^{(\ell-1)d}$. For this, one needs that the sphere $\{x_{1},...,x_{k}\}$ given by the equations \eqref{sphererepresentation} should be oriented as follows. One projects this sphere to $(x_{1},...,x_{k-1})$. We get an ellipsoid whose orientation is such that the outside normal vector taken as a first one, union the oriented tangent frame gives $(-1)^{kd}$ times the standard orientation of $\mathbb{R}^{(k-1)d}$.\footnote{We do not actually need it, but it is worth mentioning as in the original paper \cite{D_T}, the orientation of $\{x_{1},...,x_{k}\}$ has not been determined. It was just said that the orientation of $\{x_{1},...,x_{k}\}$ is such that the pairing \eqref{psi} works, see \cite[footnote 3]{D_T}.}

Now to determine the sign in front of \eqref{sum2} in Theorem \ref{maintheorem}, we consider the intersection of the cocycle \begin{tikzpicture}[baseline=-0.4cm]
\draw node [left, yshift=0.2cm] {1} (0,0) rectangle node {$1,...,\widehat{k_{1}-1},...,k_{1}+k_{2}-2$} (4.5,0.7); 
\draw (0.8,-1.2) ellipse (0.6 and 0.3) node {$k_{1}-1$};
\draw (3.2,-1.2) ellipse (1 and 0.3) node {$k_{1}+k_{2}-1$};
\draw [->] (0.8,0) -- (0.8,-0.9) node [left, yshift=0.4cm] {2};
\draw [->] (3.2,0) -- (3.2,-0.9) node [left, yshift=0.4cm] {3};
\end{tikzpicture} with $\{x_{1},...,x_{k_{1}-1},\{x_{k_{1}},...,x_{k_{1}+k_{2}-1}\}\}$ and with the right hand side of \eqref{sum2}. The corresponding cochain intersects only the last summand of \eqref{sum2}, $[x_{k_{1}+k_{2}-1},\{x_{1},...,x_{k_{1}+k_{2}-2}\}]$. The latter intersection is obtained by computing 

\begin{multline*}
\Psi([x_{k_{1}+k_{2}-1},\{x_{1},...,x_{k_{1}+k_{2}-2}\}]) =\\ \begin{tikzpicture}[baseline=-0.4cm]
\draw (0,0) rectangle node {$1,...,\widehat{k_{1}-1},...,k_{1}+k_{2}-2$} (4.5,0.7)node (2) [right, yshift=-0.2cm] {2}; 
\draw (3.2,-1.2) ellipse (0.6 and 0.3) node  {$k_{1}-1$};
\draw (-2,0.3) ellipse (1 and 0.3) node (round) {$k_{1}+k_{2}-1$};
\draw [->] (-1,0.3) -- (0,0.3) node [left, yshift=-0.3cm, xshift=-0.4cm] {1};
\draw [->] (3.2,0) -- (3.2,-0.9) node [left, yshift=0.4cm] {3};
\node [right of=2, xshift=-0.45cm, yshift=-0.2cm] {$+\cdots$};
\node [left of= round, xshift=-1.5cm] {$(-1)^{(k_{1}-2)d}$};
\end{tikzpicture}
\end{multline*}
See Section 8 in \cite{D_T}, in particular Example 8.1(a). The other summands are $(k_{1}+k_{2}-2)$-trees of different shapes (and thus do not contribute). The sign in front comes from \eqref{psi}. To get the desired intersection, first we reverse the arrow between the square vertex and the round vertex labeled $k_{1}+k_{2}-1$. This gives the sign $(-1)^{d}$. After reversing the arrow, we have the following $(k_{1}+k_{2}-2)$-tree:
\begin{equation*}
\begin{tikzpicture}[baseline=-0.4cm]
\draw node (2) [left, yshift=0.2cm] {2} (0,0) rectangle node {$1,...,\widehat{k_{1}-1},...,k_{1}+k_{2}-2$} (4.5,0.7); 
\draw (3.2,-1.2) ellipse (0.6 and 0.3) node {$k_{1}-1$};
\draw (1,-1.2) ellipse (1 and 0.3) node {$k_{1}+k_{2}-1$};
\draw [->] (1,0) -- (1,-0.9) node [left, yshift=0.4cm] {1};
\draw [->] (3.2,0) -- (3.2,-0.9) node [left, yshift=0.4cm] {3};
\node [draw=none, fill=none, left of=2, xshift=-0.5cm] {$(-1)^{(k_{1}-1)d}$};
\end{tikzpicture}.
\end{equation*}

Next, we change the order of the elements 1, 2 and 3 in the orientation set by pulling 2 in front and pushing 1 to the end. The degree of 1 and 3 is $d-1$. The degree of 2 is $(k_{1}+k_{2}-4)d$. So the obtained sign from reordering the orientation set is $(-1)^{(d-1)(k_{1}+k_{2}-4)d}\times(-1)^{(d-1)(d-1)
} = (-1)^{d-1}$. After this change, we now have the $(k_{1}+k_{2}-2)$-tree that we want: 
\begin{equation*}
\begin{tikzpicture}[baseline=-0.4cm]
\draw node (1) [left, yshift=0.2cm] {1} (0,0) rectangle node {$1,...,\widehat{k_{1}-1},...,k_{1}+k_{2}-2$} (4.5,0.7); 
\draw (0.8,-1.2) ellipse (0.6 and 0.3) node {$k_{1}-1$};
\draw (3.2,-1.2) ellipse (1 and 0.3) node {$k_{1}+k_{2}-1$};
\draw [->] (0.8,0) -- (0.8,-0.9) node [left, yshift=0.4cm] {2};
\draw [->] (3.2,0) -- (3.2,-0.9) node [left, yshift=0.4cm] {3};
\node [draw=none, fill=none, left of=1, xshift=-0.2cm] {$(-1)^{k_{1}d-1}$};
\end{tikzpicture}.
\end{equation*}

\noindent Finally, we computed the intersection 
\begin{equation}\label{rightintersection}
[x_{k_{1}+k_{2}-1}, \{x_{1},...,x_{k_{1}+k_{2}-2}\}] \; \bigcap \begin{tikzpicture}[baseline=-0.4cm]
\draw node (1) [left, yshift=0.2cm] {1} (0,0) rectangle node {$1,...,\widehat{k_{1}-1},...,k_{1}+k_{2}-2$} (4.5,0.7); 
\draw (0.8,-1.2) ellipse (0.6 and 0.3) node {$k_{1}-1$};
\draw (3.2,-1.2) ellipse (1 and 0.3) node {$k_{1}+k_{2}-1$};
\draw [->] (0.8,0) -- (0.8,-0.9) node [left, yshift=0.4cm] {2};
\draw [->] (3.2,0) -- (3.2,-0.9) node [left, yshift=0.4cm] {3};
\end{tikzpicture} =(-1)^{k_{1}d-1}.
\end{equation}

The sign in front of the last summand of \eqref{sum2} is $(-1)^{(k_{1}-1)d}\times (-1)^{(k_{1}+k_{2}-2)d} = (-1)^{(k_{2}-1)d}$. In conclusion, we obtained that the intersection with the right hand side of \eqref{sum2} is
\begin{equation} \label{finalsign}
(-1)^{k_{1}d-1}\times(-1)^{(k_{2}-1)d}=(-1)^{(k_{1}+k_{2}-1)d-1}.
\end{equation}
Next we want to check that this is the same sign that we get on the left hand side for $\{x_{1},...,x_{k_{1}-1},\{x_{k_{1}},...,x_{k_{1}+k_{2}-1}\}\}$. This cycle is the product of two spheres $S^{(k_{1}-1)d}\times S^{(k_{2}-1)d-1}$.

The chain 
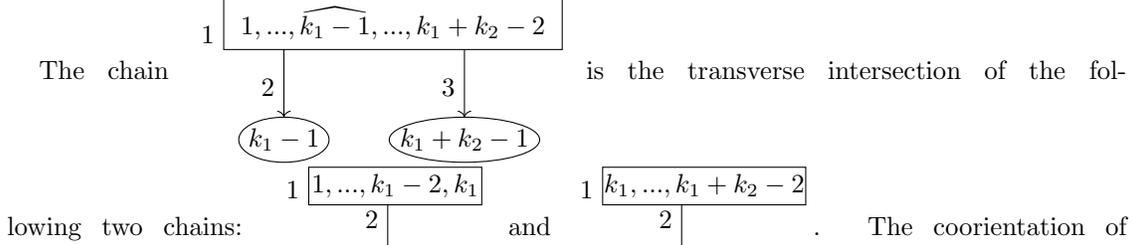
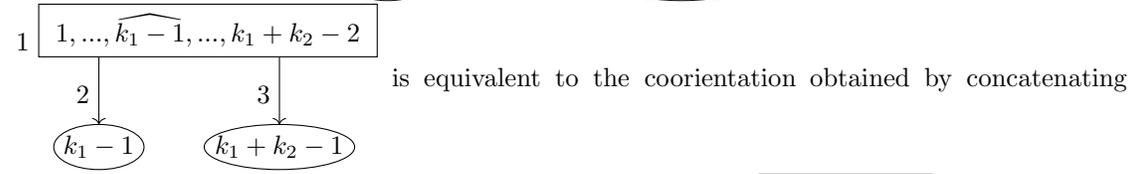
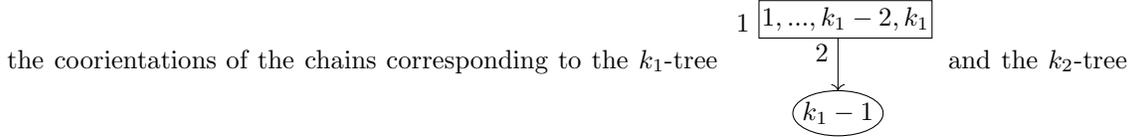
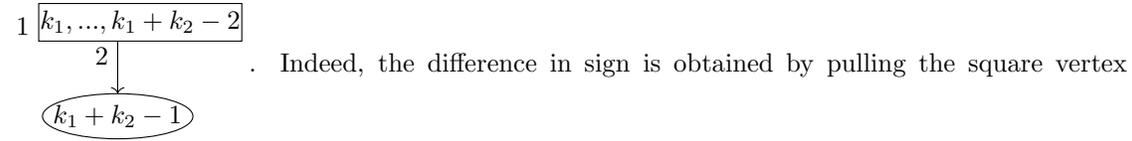
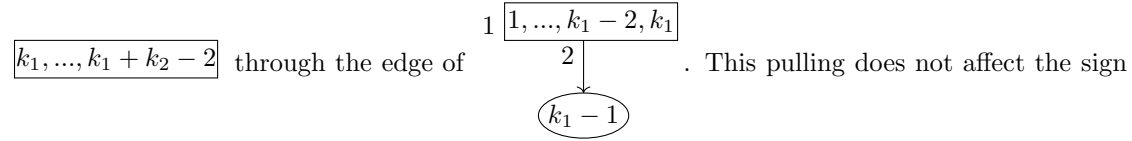
\begin{tikzpicture}[baseline=-0.4cm]
\draw node (1) [left, yshift=0.2cm] {1} (0,0) rectangle node {$1,...,\widehat{k_{1}-1},...,k_{1}+k_{2}-2$} (4.5,0.7); 
\draw (0.8,-1.2) ellipse (0.6 and 0.3) node {$k_{1}-1$};
\draw (3.2,-1.2) ellipse (1 and 0.3) node {$k_{1}+k_{2}-1$};
\draw [->] (0.8,0) -- (0.8,-0.9) node [left, yshift=0.4cm] {2};
\draw [->] (3.2,0) -- (3.2,-0.9) node [left, yshift=0.4cm] {3};
\end{tikzpicture} is the transverse intersection of the following two chains: \begin{tikzpicture} [baseline=-0.4cm]
\draw node [left, yshift=0.2cm] {1} (0,0) rectangle node {$1,...,k_{1}-2,k_{1}$} (2.3,0.5); 
\draw (1.05,-1) ellipse (0.6 and 0.3) node {$k_{1}-1$};
\draw [->] (1.05,0) -- (1.05,-0.7) node [left, yshift=0.5cm] {2};
\end{tikzpicture} and \begin{tikzpicture} [baseline=-0.4cm]
\draw node [left, yshift=0.2cm] {1} (0,0) rectangle node {$k_{1},...,k_{1}+k_{2}-2$} (2.7,0.5); 
\draw (1.05,-1) ellipse (1 and 0.3) node {$k_{1}+k_{2}-1$};
\draw [->] (1.05,0) -- (1.05,-0.7) node [left, yshift=0.5cm] {2};
\end{tikzpicture}. The coorientation of \begin{tikzpicture}[baseline=-0.4cm]
\draw node (1) [left, yshift=0.2cm] {1} (0,0) rectangle node {$1,...,\widehat{k_{1}-1},...,k_{1}+k_{2}-2$} (4.5,0.7); 
\draw (0.8,-1.2) ellipse (0.6 and 0.3) node {$k_{1}-1$};
\draw (3.2,-1.2) ellipse (1 and 0.3) node {$k_{1}+k_{2}-1$};
\draw [->] (0.8,0) -- (0.8,-0.9) node [left, yshift=0.4cm] {2};
\draw [->] (3.2,0) -- (3.2,-0.9) node [left, yshift=0.4cm] {3};
\end{tikzpicture} is equivalent to the coorientation obtained by concatenating the coorientations of the chains corresponding to the $k_{1}$-tree \begin{tikzpicture} [baseline=-0.4cm]
\draw node [left, yshift=0.2cm] {1} (0,0) rectangle node {$1,...,k_{1}-2,k_{1}$} (2.3,0.5); 
\draw (1.05,-1) ellipse (0.6 and 0.3) node {$k_{1}-1$};
\draw [->] (1.05,0) -- (1.05,-0.7) node [left, yshift=0.5cm] {2};
\end{tikzpicture} and the $k_{2}$-tree \begin{tikzpicture} [baseline=-0.4cm]
\draw node [left, yshift=0.2cm] {1} (0,0) rectangle node {$k_{1},...,k_{1}+k_{2}-2$} (2.7,0.5); 
\draw (1.05,-1) ellipse (1 and 0.3) node {$k_{1}+k_{2}-1$};
\draw [->] (1.05,0) -- (1.05,-0.7) node [left, yshift=0.5cm] {2};
\end{tikzpicture}. Indeed, the difference in sign is obtained by pulling the square vertex \begin{tikzpicture}[baseline=0.1cm]
\draw (0,0) rectangle node {$k_{1},...,k_{1}+k_{2}-2$} (2.7,0.5);
\end{tikzpicture} through the edge of \begin{tikzpicture} [baseline=-0.4cm]
\draw node [left, yshift=0.2cm] {1} (0,0) rectangle node {$1,...,k_{1}-2,k_{1}$} (2.3,0.5); 
\draw (1.05,-1) ellipse (0.6 and 0.3) node {$k_{1}-1$};
\draw [->] (1.05,0) -- (1.05,-0.7) node [left, yshift=0.5cm] {2};
\end{tikzpicture}. This pulling does not affect the sign since the degree of \begin{tikzpicture}[baseline=0.1cm]
\draw (0,0) rectangle node {$k_{1},...,k_{1}+k_{2}-2$} (2.7,0.5);
\end{tikzpicture} is a multiple of $d$ and the degree of the edge is $d-1$. By \eqref{psi}, the intersection of $S^{(k_{1}-1)d-1}$ (the first factor of $\{x_{1},...,x_{k_{1}-1},\{x_{k_{1}},...,x_{k_{1}+k_{2}-1}\}\}$) with 
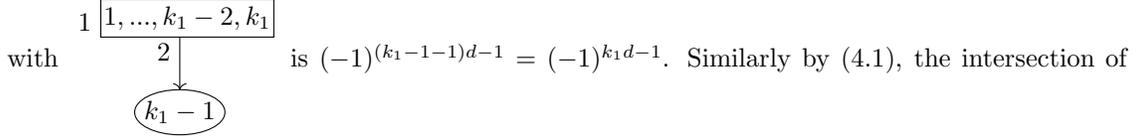
\begin{tikzpicture} [baseline=-0.4cm]
\draw node [left, yshift=0.2cm] {1} (0,0) rectangle node {$1,...,k_{1}-2,k_{1}$} (2.3,0.5); 
\draw (1.05,-1) ellipse (0.6 and 0.3) node {$k_{1}-1$};
\draw [->] (1.05,0) -- (1.05,-0.7) node [left, yshift=0.5cm] {2};
\end{tikzpicture} is $(-1)^{(k_{1}-1-1)d-1} = (-1)^{k_{1}d-1}$. Similarly by \eqref{psi}, the intersection of $S^{(k_{2}-1)d-1}$ (the second factor of $\{x_{1},...,x_{k_{1}-1},\{x_{k_{1}},...,x_{k_{1}+k_{2}-1}\}\}$) and 
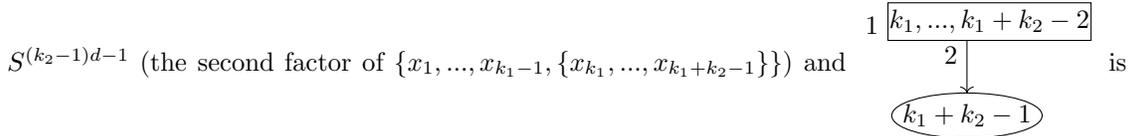
\begin{tikzpicture} [baseline=-0.4cm]
\draw node [left, yshift=0.2cm] {1} (0,0) rectangle node {$k_{1},...,k_{1}+k_{2}-2$} (2.7,0.5); 
\draw (1.05,-1) ellipse (1 and 0.3) node {$k_{1}+k_{2}-1$};
\draw [->] (1.05,0) -- (1.05,-0.7) node [left, yshift=0.5cm] {2};
\end{tikzpicture} is $(-1)^{k_{2}d-d-1}$. So the total sign is $(-1)^{(k_{1}+k_{2}-1)d-1}$, which is exactly the same as \eqref{finalsign}. Therefore, the sign $(-1)^{(k_{1}-1)d}$ in front of \eqref{sum2} is correct. 

\nocite{*}
\bibliography{filteredoperads} 
\bibliographystyle{plain}

\end{document}